\documentclass{article}

\usepackage{a4wide,amsmath,amsthm,amssymb,enumerate,epsfig,color}

\def\<{\langle}
\def\>{\rangle}

 \def\backsl{\backslash}

\def\C{{\mathbb C}}
\def\N{{\mathbb N}}
\def\Z{{\mathbb Z}}
\def\R{{\mathbb R}}

\def\co{\colon \thinspace}

 \def\wtil{\widetilde}
\renewcommand{\phi}{\varphi}
\newtheorem{theorem}{Theorem}[section]
\newtheorem{lemma}[theorem]{Lemma}
\newtheorem{proposition}[theorem]{Proposition}
\newtheorem{corollary}[theorem]{Corollary}

\newtheorem{definition}[theorem]{Definition}
\newtheorem{conjecture}[theorem]{Conjecture}

\title{Reducible braids and Garside theory}
\author{Juan Gonz\'alez-Meneses\footnote{Partially supported under Australian Research Council's Discovery Projects funding scheme (project number DP1094072), the Spansih Projects MTM2007-66929, P09-FQM-5112 and FEDER.} \and Bert Wiest}
\date{August 2, 2010}

\begin{document}

\maketitle
{\abstract We show that reducible braids which are, in a Garside-theoretical sense, as simple as possible within their conjugacy class, are also as simple as possible in a geometric sense. More precisely, if a braid belongs to a certain subset of its conjugacy class which we call the stabilized set of sliding circuits, and if it is reducible, then its reducibility is geometrically obvious: it has a round or almost round reducing curve. Moreover, for any given braid, an element of its stabilized set of sliding circuits can be found using the well-known cyclic sliding operation. This leads to a polynomial time algorithm for deciding the Nielsen-Thurston type of any braid, modulo one well-known conjecture on the speed of convergence of the cyclic sliding operation.}


\section{Introduction}

There are currently two known approaches to the problem of determining
algorithmically the Nielsen-Thurston type of a given braid, i.e.\
deciding whether it is reducible, periodic, or pseudo-Anosov
\cite{FLB,CB,FarbMargalit}. Since periodicity of braids is fast
and easy to detect \cite{GW}, the main difficulty is to determine
whether a given braid is reducible.

One approach is due to Bestvina and Handel~\cite{BH}, and uses the
theory of train tracks (see also~\cite{Los}). The algorithmic complexity
of the Bestvina-Handel algorithm is still mysterious -- this is
particularly regrettable since it seems to be fast in practice,
at least generically.

The second approach, which was initiated by Benardete, Guti\'errez and
Nitecki~\cite{BGN}, and developed by Lee and Lee~\cite{LL}, uses the Garside structure, as exposed in~\cite{EM},
on the braid group. Indeed, it is shown in~\cite{BGN} that round
reduction curves are preserved by cycling and decycling. As a consequence, if a
given braid~$x\in B_n$ is reducible, then there
is at least one element of its super summit set~\cite{EM} which has a round
reduction curve, and whose reducibility is thus easy to to detect.
The drawback of this approach is that the algorithm has to compute
the complete super summit set of~$x$, and this is is very slow~\cite{Gebhardt}.

In order to have any hope of obtaining a polynomial time algorithm
from the second approach, we would need to replace the super summit set of $x$ with another set satisfying the following properties: (1) It is an invariant of the conjugacy class of $x$, (2) an element in this subset can be computed efficiently, and (3) for {\it every} element in this subset, the reducibility or irreducibility can be detected rapidly. Super summit sets satisfy the first two properties, but not the third.

In the special case of the four-strand braid group, the super summit set
can actually do the job~\cite{CalvezWiest}. In the general case of the braid
group~$B_n$ (with $n\in\N$), the ultra summit set defined in~\cite{Gebhardt} can do the job, but only under certain conditions. It is shown in~\cite{LL} that if a braid is reducible and the {\it external component} is simpler (from the Garside theoretical point of view) than the whole braid, then one can rapidly detect reducibility of any given element in its ultra summit set, as every element in this set has a round reduction curve. Hence, under this hypothesis, the ultra summit set satisfies (1) and (3) above. It is a well-known conjecture~\cite{BGG1} that it also satisfies (2).

The aim of the present paper is to construct a subset of any conjugacy class which satisfies (1) and (3) above, and is conjectured to also satisfy (2), just like Lee and Lee's subset~\cite{LL}, but without their technical hypothesis. In particular, we prove the existence of a polynomial time algorithm for deciding the reducibility or irreducibility of a given braid, modulo a well-known conjecture (Conjecture~\ref{C:no of slidings}), again concerning (2) above, which we leave open.

Where Benardete, Guti\'errez and Nitecki talk about round curves, we
have to admit a somewhat larger family of reducing curves which we call
\emph{almost round curves}.
Also, the subset of the conjugacy class for which our result holds is neither
the super summit set nor the ultra summit set, but a slightly more
complicated class, which we call the $m$ times stabilised set of sliding circuits, denoted $SC^{[m]}(x)$, where~$m$ is a positive integer.

We will show that one can conjugate a given element~$x$
of~$B_n$ to an element in $SC^{[m]}(x)$, by applying iteratively a
special kind of conjugation called {\it cyclic sliding}. This iterated
cyclic sliding procedure is a Garside-theoretic tool which simplifies
(from an algebraic point of view) the braid within its conjugacy class,
and which has already been used to solve the conjugacy problem in braid
groups and Garside groups \cite{GG1,GG2}.

Further, we will show the following result (where~$\Delta$ denotes the
half twist of all strands, so that $||\Delta||=n(n-1)/2$):

{\bf Theorem~\ref{T:main} }
{\it Let $x\in B_n$ be a non-periodic, reducible braid. There is some $m\leqslant ||\Delta||^3-||\Delta||^2$ such that every element $y\in SC^{[m]}(x)$ admits an essential reduction curve which is either round or almost round.
}

Theorem~\ref{T:main} is telling us that cyclic sliding not only simplifies braids from the algebraic, but also from the geometric point of view, since the reduction curves, which can be terribly tangled in~$x$, become either round or almost round after iterative applications of cyclic slidings.

Moreover, we prove that it can be efficiently checked whether there are
round or almost round curves which are preserved by a braid $y$ like in the statement of Theorem~\ref{T:main}. More precisely, invariant round curves can be efficiently detected by~\cite{BGN}. For almost round curves the situation is not the same: as the number of such curves grows exponentially with respect to the number of strands, it is not a good idea to try to check them one by one. To bypass this difficulty, we show the following particular case:

{\bf Theorem \ref{T:round_almostround}}
{\it There is an algorithm which decides whether a given positive braid~$x$ of length~$\ell$ with~$n$ strands preserves an almost-round curve whose interior strands do not cross. Moreover, this algorithm takes time $O(\ell \cdot n^4)$.}

Notice that Theorem~\ref{T:round_almostround} cannot immediately be applied to detect the reduction curves promised by Theorem~\ref{T:main}, for two reasons: firstly, none of these curves are necessarily $x$-invariant (they may be permuted by~$x$), and secondly, even if they were, there would be no guarantee that their interior strands do not cross. Moreover, $x$ is not necessarily positive (although this can be easily achieved just by multiplying $x$ by a suitable power of $\Delta^2$). There is, however, a situation which can be reduced to the cases that can be checked using Theorem ~\ref{T:round_almostround}. This is the situation where the given braid is {\it rigid}~\cite{BGG1}.

{\bf Theorem~\ref{T:rigid_case}.}
{\it Let $\beta\in B_n$ be a non-periodic, reducible braid which is rigid. Then there is some positive integer $k\leqslant n$ such that one of the following conditions holds:
\begin{enumerate}
\item $\beta^k$ preserves a round essential curve, or

\item $\inf(\beta^k)$ and $\sup(\beta^k)$ are even, and either $\Delta^{-\inf(\beta^k)}\beta^k$ or $\beta^{-k}\Delta^{\sup(\beta^k)}$ is a positive braid which preserves an almost round essential reduction curve whose corresponding interior strands do not cross.
\end{enumerate}
In particular, some essential reduction curve for $\beta$ is either round or almost round.
}

The power $k\leqslant n$ in the above statement is needed to pass from invariant {\it families of curves} to invariant {\it curves}, which is what is detected in Theorem~\ref{T:round_almostround}, and also to assure that $\inf(\beta^k)$ and $\sup(\beta^k)$ are even. We then see that if the braid $\beta$ under study is rigid and admits essential reduction curves, we can find them in one of the following two ways: If one these curves is round, we can apply the well known algorithm in~\cite{BGG2}. Otherwise, we will find them by applying Theorem~\ref{T:round_almostround} to $\Delta^{-\inf(\beta^k)}\beta^k$ (where $\inf(\beta^k)$ is even) and to $\beta^{-k}\Delta^{\sup(\beta^k)}$ (where $\sup(\beta^k)$ is even) for $k=1,\ldots, n/2$, as these braids have the same essential reduction curves as $\beta$.

The next aim is to construct, for any given $y\in SC^{[N]}(x)$, a rigid braid whose reducing curves are also reducing curves of~$y$. This serves two purposes at once: it allows us to use Theorem~\ref{T:round_almostround} to search for reducing curves in polynomial time, and it also gives the key to proving Theorem~\ref{T:main}.

In order to do so, we will, for every braid $y\in SC^{[N]}(x)$, define its {\it preferred conjugator} $P(y)$, which commutes with~$y$. We will prove:

{\bf Lemma~\ref{L:main_detailed}.}
{\it Let $x\in B_n$ be a non-periodic, reducible braid. Let $N=||\Delta||^3-||\Delta||^2$. For every element $y\in SC^{[N]}(x)$ there is some $m\leqslant N$ such that either $y^m$ is rigid, or $P(y^m)$ is rigid, admits essential reduction curves, and all its essential reduction curves are essential reduction curves of~$y$.}

From the above results, we obtain the following algorithm to determine whether a given element of $B_n$ is periodic, reducible or pseudo-Anosov:

\noindent {\bf Algorithm 1.} To determine the geometric type of a braid.

Input: $x\in B_n$.
\begin{enumerate}

 \item If $x^{n-1}$ or $x^n$ is a power of~$\Delta$, return `{\it $x$ is periodic}' and stop.

 \item Compute an element $y\in SC^{[N]}$, where $N=||\Delta||^3-||\Delta||^2$.

 \item If $y$ preserves a family of round curves, return `{\it $x$ is reducible, non-periodic}' and stop.

 \item For $m=1,\ldots,N$ do the following:

 If either $y^m$ is rigid or $P(y^m)$ is rigid, apply the algorithm in Theorem~\ref{T:round_almostround} to the braids mentioned in Theorem~\ref{T:rigid_case}(2), with $\beta=y^m$ or $\beta=P(y^m)$, respectively. If an almost round reduction curve is found, return `{\it $x$ is reducible, non-periodic}' and stop.

 \item Return `{\it $x$ is pseudo-Anosov}'.

\end{enumerate}

The computational complexity of each step of this algorithm is bounded by a polynomial in the length and the number of strands of $x$, with one exception: the second step of this algorithm (conjugating $x$ to $y\in SC^{[N]}(x)$) is not currently known to be doable in polynomial time, but it is conjectured to be so (c.f.\ Conjecture~\ref{C:no of slidings}).

The plan of the paper is as follows. In Section~\ref{S:reduction_curves}
we introduce the basic notions of reducible braids and reduction curves,
including the proof of Theorem~\ref{T:round_almostround}, and of
Theorem~\ref{T:main} in the case where the interior braid is trivial.
In Section~\ref{S:sliding_circuits} we switch to the algebraic viewpoint, explaining the notion of cyclic sliding and sliding circuits, and introducing the set $SC^{[m]}(x)$. Exploring the relation between sliding circuits and the powers of a braid, in Section~\ref{S:SC and powers}, we show how to compute one element in $SC^{[m]}(x)$ for every~$x$ and~$m$. We then proceed to study, in Section~\ref{S:SC and curves}, the relation between the reduction curves, on the geometric side, and the sets of sliding circuits, on the algebraic side. At the end of this section, we show that Theorem~\ref{T:main} holds in general if it holds for the special case of {\it rigid} braids.  Section~\ref{S:reducible rigid braids} treats the case of reducible rigid braids, finishing the proof of Theorem~\ref{T:main} by showing that if a rigid, reducible braid has some interior braid which is pseudo-Anosov, then its corresponding reduction curve is round.

{\bf Acknowledgements:} We wish to thank Volker Gebhardt for many useful discussions on this and related problems.


\section{Round and almost round reduction curves}\label{S:reduction_curves}

\subsection{Definitions and notations}

\subsubsection{Canonical reduction system and complexity of curves}

Let~$B_n$ be the braid group on~$n$ strands, where we fix as base points
the set $P_n=\{1,\ldots, n\}\in \mathbb C$. Every element $x\in B_n$ can
be seen as an automorphism of $D_n=D^2\backslash P_n$, where~$D^2$
denotes the disk in~$\mathbb C$ with diameter $[0,n+1]$. Therefore~$x$
induces an action on the isotopy classes of 1-manifolds in~$D_n$.

We will consider the action of braids on isotopy classes of simple curves {\it from the right}. That is, we will denote the isotopy class of a
simple curve~$\mathcal C$ by~$[\mathcal C]$, and we will write
$[\mathcal C]^x$, meaning the isotopy class of the curve
obtained from~$\mathcal C$ after applying~$x$ considered as an
automorphism of the $n$-times punctured disk. By abuse of vocabulary, we shall often say ``curves'' when we really mean ``isotopy classes of curves''. However, we shall carefully distinguish the notations~$\mathcal C$ and~$[\mathcal C]$.

A simple closed curve~$\mathcal C$ in $D^2\backslash P_n$ is
said to be \emph{non-degenerate} if it encloses more than one and less
than~$n$ points of~$P_n$, and it is said to be \emph{round} if it is
homotopic to a geometric circle. It is clear that non-degeneracy and roundness are properties which depend only on the isotopy class of a curve, so we can naturally say that some isotopy  class $[\mathcal C]$ is non-degenerate, or is round. A braid $x\in B_n$ is said to be
\emph{reducible} if $[\mathcal C]^{(x^m)}=[\mathcal C]$, for some positive
integer~$m$ and some non-degenerate curve~$\mathcal C$. Such a curve $\mathcal C$ is said to be a {\it reduction curve} for $x$. We say that a reduction curve ${\mathcal C}$ is {\it essential} if every other reduction curve for $x$ can be isotoped to have empty intersection with $\mathcal C$~\cite{BLM}.

The set of isotopy classes of essential reduction curves of a braid~$x$ is called the
\emph{canonical reduction system} of~$x$, and is denoted~$CRS(x)$.
It is well known that $CRS(x)=\emptyset$ if and only if~$x$ is either
periodic or pseudo-Anosov~\cite{BLM}. In other words, $CRS(x)\neq \emptyset$ if and
only if~$x$ is reducible and non-periodic. Since it is very easy to
determine whether a given braid $x\in B_n$ is periodic (it suffices to
check if either~$x^{n-1}$ or~$x^{n}$ is equal to a power of the half twist~$\Delta$),
the question of determining the geometric type of a braid reduces to
the study of its canonical reduction system. We will then be interested
in reducible, non-periodic braids, and in their essential reduction curves.

We will say that a non-degenerate simple curve~$\mathcal C$
in~$D^2\backslash P_n$ is \emph{almost round} if there exists a
simple element~$s$ (a permutation braid) such that~$[\mathcal C]^s$ is
round. This is equivalent to saying that~$\mathcal C$ can be isotoped
in $D^2\backslash P_n$ to a curve whose projection to the real
line has exactly one local maximum and one local minimum.

There is an alternative characterization of almost round curves which
will also allow us to introduce a notion of \emph{complexity} of a
simple closed curve in the punctured disc. Notice that a
curve~$[\mathcal C]$ can always be transformed into a round curve by a
suitable automorphism of the punctured disc, that is, by a suitable
braid~$y$. Since the full twist~$\Delta^2$ preserves any given curve,
it follows that~$\Delta^{2k}\beta$ also transforms~$[\mathcal C]$ into
a round curve, for every integer~$k$. Hence we can assume that~$y$
is a positive braid, as every braid becomes positive after multiplication
by a sufficiently high power of~$\Delta^2$. It is shown in~\cite{LL} that given a family $\mathcal F$ of mutually disjoint simple closed curves in $D^2\backslash P_n$, there is a {\it unique} positive braid $y\in B_n$ such that $[\mathcal F]^y$ is a family of round curves, and such that~$y$ has minimal length among all positive braids satisfying this property (actually $y$ is a prefix of any other positive braid satisfying this property). This braid~$y$ is called
the {\it minimal standardizer} of~$\mathcal F$. If~$\mathcal F$ consists of a single curve~$\mathcal C$, we will call~$y$ the {\it minimal standardizer} of $\mathcal C$.

Now recall that the \emph{simple braids} (or permutation braids) are those positive braids
for which every pair of strands cross at most once, and that~$\ell(x)$,
the canonical length of a braid~$x$, is the minimal number of simple factors
into which~$x$ can be decomposed, not counting factors equal to the
half twist~$\Delta$ -- see also Section~\ref{S:sliding_circuits}.
Alternatively, the canonical length $\ell(x)$ is the number of
factors different from~$\Delta$ in the left normal form of~$x$.

\begin{definition}
Given a simple closed curve~$\mathcal C$ in the punctured disc, we
define the \emph{complexity} of~$\mathcal C$ to be the canonical length of the minimal standardizer of $\mathcal C$.
\end{definition}

In other words, the complexity of $\mathcal C$ is the smallest
possible canonical length of a positive braid sending~$[\mathcal C]$
to a round curve.
Notice that this definition could be equivalently expressed the
other way around: the complexity of~$\mathcal C$ is the smallest possible
canonical length of a positive braid sending a round curve to~$[\mathcal C]$.

The curves of complexity 0 are the round curves, and the curves of
complexity~1 are those which become round by the action of a simple
element: these are precisely the almost round, not round curves.

\subsubsection{Decomposition of a braid along a family of curves}

Reduction curves allow us to decompose a braid into simpler braids.
In fact, several procedures for specifying such a decomposition are
conceivable, but we shall
use the procedure given in~\cite{G2009}, which we briefly explain now.

Let $x\in B_n$, and let $\mathcal F$ be a family of disjoint simple closed curves in $D^2\backslash P_n$. 
Let $y$ be the minimal standardizer of $\mathcal F$, and let $\widehat x = y^{-1}xy =: x^y$ and $\widehat{\mathcal F}=\mathcal F^y$. Notice that if $x$ preserves $[\mathcal F]$, then $\widehat x$ preserves $[\widehat{\mathcal F}]=[\mathcal F]^{y}$, which is a family of round curves.
However, even if $x$ does not preserve $[\mathcal F]$, it can still happen that $\widehat x$ sends $[\widehat{\mathcal F}]=[\mathcal F]^{y}$ to a family of round curves (not necessarily~$[\widehat{\mathcal F}]$ itself). In this case we can define for every curve $\mathcal C\in \mathcal F\cup \{\partial(D^2)\}$, a braid $\beta_{[\mathcal C\in \mathcal F]}$, called the \emph{component of $x$ associated to $\mathcal C$ in} $\mathcal F$, as follows.

For every subset $I\subset \{1,\ldots n\}$, we can define the {\it subbraid} $(\widehat x)_{I}$ to be the braid on $\#(I)$ strands obtained from $\widehat x$ by keeping only those strands which start at~$I$.
Notice that this yields a well-defined element of~$B_{\#(I)}$, even if
the strands starting at $I$ do not end at $I$ -- we just require the strands of $(\widehat x)_I$ to cross in the same way as the strands in $\widehat x$ starting at $I$, for details see~\cite{G2009}.

Now given a curve $\mathcal C\in \mathcal F\cup \{\partial(D^2)\}$, let $X_{\mathcal C}$ be the only connected component of $D^2\backslash \mathcal F$ which is enclosed by $\mathcal C$, and such that $\mathcal C\subset \overline{X_{\mathcal C}}$. Then define $D_{\mathcal C}= X_{\mathcal C}\cup \mathcal C$, which is homeomorphic to a punctured disc. 
Notice that $\overline{D_{\mathcal C}}\backslash D_{\mathcal C}$ is a family of points and curves, namely the outermost curves enclosed by $\mathcal C$, and the points which are enclosed by $\mathcal C$ but not enclosed by the mentioned curves.

Similarly, we let $X_{\widehat{\mathcal C}}$ be the only connected component of $D^2\backslash \widehat{\mathcal F}$ which is enclosed by $\widehat{\mathcal C}$, and such that $\widehat{\mathcal C}\subset \overline{X_{\widehat{\mathcal C}}}$. Then define $D_{\widehat{\mathcal C}}= X_{\widehat{\mathcal C}}\cup \widehat{\mathcal C}$. This is a closed round disk with some points and some closed round disks removed from its interior.
Also, $\overline{D_{\widehat{\mathcal C}}}\backslash D_{\widehat{\mathcal C}}$ is a family of points and round curves.


\begin{definition}{\rm \cite{G2009}}
Let $x\in B_n$ and let $\mathcal F$ be a family of disjoint simple closed curves in $D^2\backslash P_n$, whose minimal standardizer is $y$. Let $\widehat x = y^{-1}x y$, and suppose that $\widehat x$ sends $[\widehat{\mathcal F}]=[\mathcal F]^{y}$ to a family of round curves. Let $\mathcal C\in \mathcal F\cup \{\partial(D^2)\}$.
Let $I\subset \{1,\ldots,n\}$ consist of the indices of\vspace{-1mm}
\begin{itemize}
\item those punctures that appear in $\overline{D_{\widehat{\mathcal C}}}\backslash D_{\widehat{\mathcal C}}$, and\vspace{-1mm}
\item for each curve in $\overline{D_{\widehat{\mathcal C}}}\backslash D_{\widehat{\mathcal C}}$, exactly one puncture chosen arbitrarily among the punctures enclosed by that curve.\vspace{-1mm}
\end{itemize}
Then we define $x_{[\mathcal C\in \mathcal F]}$, the component of $x$ associated to $\mathcal C$ in $\mathcal F$, as the subbraid $(\widehat x)_{I}$.  If $\mathcal F=CRS(x)$, the mentioned component is just denoted $x_{\mathcal C}$.
\end{definition}

We remark that $x_{\partial(D^2)}$ is usually called the \emph{external braid} associated to $x$, and is denoted $x^{ext}$.



\subsection{Canonical reduction curves of reducible, positive braids
with trivial interior braids are either round or almost round}
\label{SS:CRSAlmostRound}

The aim of this section is to prove the following result:

\begin{proposition}\label{P:trivial interior -> almost round}
If $\mathcal C$ is an essential reduction curve for a positive braid $x$, with $[\mathcal C]^x=[\mathcal C]$, and the strands of $x$ enclosed by $\mathcal C$ do not cross each other, then $\mathcal C$ is either round or almost round.
\end{proposition}

In order to show this result, it suffices to prove that such a curve~$\mathcal C$ cannot be of complexity two, i.e., it cannot be the result of a round curve after the action of a braid of canonical length two, without being round or almost round.

Our first aim is to understand what a curve of complexity two looks
like (for detailed discussion of more general questions see~\cite{W}).
We shall first study smooth arcs
$\alpha\co I \to D^2$ in the disk~$D^2$ defined on the
unit interval~$I=[0,1]$; we shall restrict our attention to smooth
arcs~$\alpha$ which start and end in puncture points, which may also
traverse some puncture points, but whose tangent direction is
horizontal and pointing to the right, at every puncture point. For brevity, we shall call them
\emph{arcs traversing some puncture points horizontally}.

When studying diffeotopy classes of such arcs, we shall always mean
diffeotopies through families of arcs which are all supposed to traverse
the puncture points horizontally.
We shall say that a simple closed curve or an arc traversing some puncture
points horizontally is \emph{reduced} if it
has the minimal possible number of intersections with the horizontal
line, and also the minimal possible number of vertical tangencies
in its diffeotopy class.

The action of the braid group on the set of diffeotopy classes of arcs
traversing some puncture points horizontally, specifically of a braid $x$
on an arc~$\alpha$, is defined as follows:
$x$~induces a puncture dance, which in turn can be extended to a
diffeotopy of~$\alpha$ in such a way that at every moment the
intersection of the arc with the punctures is horizontal.
At the end of the dance we obtain a new arc traversing some puncture
points horizontally, which is well-defined up to diffeotopy.
This is~$\alpha^x$.

For an arc~$\alpha$ traversing some puncture points horizontally,
we define the \emph{tangent direction function}
$t_\alpha\co I\to \R / 2\Z$ as the angle of the tangent direction
of~$\alpha$ against the horizontal, divided by~$-\pi$. In particular,
if the arc goes straight to the right in~$\alpha(t)$, then
$t_\alpha(t)=0+2\Z$, if it goes straight down then
$t_\alpha(t)=\frac{1}{2}+2\Z$, and if it goes to the left then
$t_\alpha(t)=1+2\Z$.

For every arc traversing some puncture points horizontally, we have a
unique lifting of the function~$t_\alpha$ to a function
$\wtil{t}_\alpha\co I\to \R$ with $\wtil{t}_\alpha(0)=0$.
Finally, if $r\co \R\to\Z$ denotes the rounding function, which
sends every real number to the nearest integer (rounding \emph{down}
$n+\frac{1}{2}$), then we define the function
$$
\tau_\alpha\co I\to \Z, \ t\mapsto r\circ\wtil{t}_\alpha(t)
$$
which one might call the rounded lifted tangent direction function.

Notice that, if $\alpha$ is an arc such that $\tau_\alpha$ takes the
value~$0$ in a neighbourhood of the points where the arc traverses a
puncture, then the same is true for its image~$\alpha^x$ under the
action of any braid.

In order to be able to characterize reduction curves of complexity zero,
one, and two, we give now a detailed
description of the puncture dance associated to a positive permutation
braid. In a first step, the punctures make a small vertical movement,
with the puncture in position $k\in\Z$ moving to position
$k-k\cdot\epsilon\cdot i \in\C$, for some small $\epsilon>0$. In a
second step, the punctures make a horizontal movement, permuting their
$\R$-coordinates. In a third step, the punctures make again a small
vertical movement, lining them back up on the real line.

Now, reduction curves~${\mathcal C}$ of complexity zero can be characterized
as curves enclosing an arc which lies entirely in the real line, and
which traverses all the punctures in the interior of~${\mathcal C}$. Notice that $\mathcal C$ can be seen as the boundary of a regular neighborhood of this arc.

Suppose now that a curve ${\mathcal C}$ has complexity one. Then it is obtained from a round curve ${\mathcal C}_0$ by the action of a simple braid $s$. We can assume that the punctures enclosed by $\mathcal C_0$ (which are consecutive) do not cross in $s$, as those crossings could be removed from $s$ without modifying its action on $\mathcal C_0$. Hence, from the above description of positive permutation braids, we see that reduction curves~${\mathcal C}$ of complexity one can be characterized as follows: there exists a smooth arc~$\alpha$ disjoint from~${\mathcal C}$, traversing all the punctures in the interior component of $D^2\backsl {\mathcal C}$ horizontally such that~$\tau_\alpha$ is the constant function~$0$. (We are going to say such an arc is \emph{almost horizontal}.)

The action by a positive permutation braid
transforms an arc~$\alpha$ with $\tau_\alpha\equiv 0$ into an
arc~$\alpha'$ which, after reduction,
has the following property: by an isotopy of~$D^2$ that moves
the~$n$ puncture points only in the vertical direction up or down,
$\alpha'$ can be transformed into an arc whose imaginary coordinate is
monotonically decreasing.
Therefore, reduction curves~${\mathcal C}$ of complexity two can be
characterized as follows: there exists a smooth arc~$\alpha'$ disjoint
from~${\mathcal C}$ but traversing horizontally all the punctures in the
interior component of~$D^2\backsl {\mathcal C}$, such that~$\tau_{\alpha'}$
only takes the values~$0$ and~$1$ (for a more detailed proof see~\cite{W}).

One important property is that if a braid $x$ preserves a curve $\mathcal C$ of complexity 2, and the strands inside $\mathcal C$ do not cross in $x$, then the mentioned arc is invariant by $x$:

\begin{lemma}\label{L:invariant_arc}
Let $x\in B_n$ and let $\mathcal C$ be a curve such that $[\mathcal C]^x=[\mathcal C]$. Suppose that the strands enclosed by $\mathcal C$ do not cross in $x$. Let $\alpha$ by an arc traversing horizontally some punctures enclosed by~${\mathcal C}$. Then $\alpha^x=\alpha$.
\end{lemma}

\begin{proof}
Let $\mathcal C_0$ be a round curve and let $y\in B_n$ be such that $[\mathcal C]^y=[\mathcal C_0]$.  Consider the braid $z=y^{-1}xy$, and the arc $\alpha^y$. Notice that $[\mathcal C_0]^z=[\mathcal C_0]^{y^{-1}xy} = [\mathcal C]^{xy} = [\mathcal C]^{y}= [\mathcal C_0]$. Hence $z$ preserves the round curve $\mathcal C_0$. Moreover, as the punctures enclosed by $\mathcal C$ do not cross in $y$, we can find a representative of $z$ in which the punctures enclosed by $\mathcal C_0$ do not cross. This implies that $z$ can be represented by a homeomorphism of the punctured disc whose restriction to the component enclosed by $\mathcal C_0$ is trivial. As $\alpha^y$ is a curve enclosed by $\mathcal C_0$, one has $(\alpha^y)^z=\alpha^y$ and then $\alpha^x = (\alpha^{yz})^{y^{-1}}= (\alpha^y)^{y^{-1}} = \alpha$, as we wanted to show.
\end{proof}

We saw above that a curve $\mathcal C$ of complexity 2 admit a smooth arc~$\alpha'$ disjoint
from~${\mathcal C}$ but traversing horizontally all the punctures in the
interior component of~$D^2\backsl {\mathcal C}$, such that~$\tau_{\alpha'}$
only takes the values~$0$ and~$1$. If $x$ is a braid preserving $\mathcal C$ in which the strands enclosed by $\mathcal C$ do not cross, the above lemma shows that the smooth arc $\alpha'$ is preserved by $x$. We shall call such an arc a \emph{descending invariant arc}.
Notice that $\mathcal C$ is the boundary of a regular neighborhood of $\alpha'$.

\begin{lemma}\label{L:maxincreases}
If~$x$ is a positive braid, if~$\alpha$ is an arc traversing
horizontally some puncture points, and if~$\alpha^x$ is its
reduced image under the action of~$x$, then
$$
\max_{t\in I}\,\tau_{\alpha^x}(t) \geqslant
\max_{t\in I}\,\tau_\alpha(t)
\text{ \ \ and \ \ }
\min_{t\in I}\,\tau_{\alpha^x}(t) \geqslant
\min_{t\in I}\,\tau_\alpha(t)
$$
\end{lemma}

\begin{proof}
It suffices to prove this result for $x=\sigma_i$, a single Artin
generator. It is an easy observation that for every~$t_0$ in~$I$,
we have $t_{\alpha^{\sigma_i}}(t_0)=t_\alpha(t_0)$ or
$t_{\alpha^{\sigma_i}}(t_0)=t_\alpha(t_0)+1$.
Some examples are given in Figure~\ref{F:genrlabel}.
\end{proof}

\begin{figure}[htb]
\centerline{\input{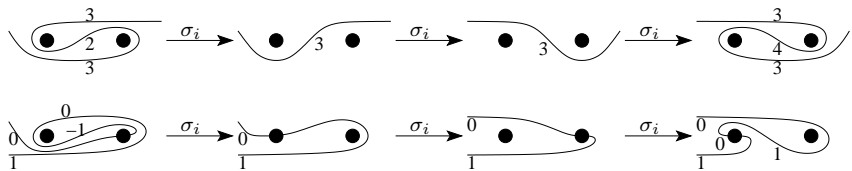}}
\vspace{-2mm}
\caption{The labels, which represent the values of the function~$t_\alpha$,
can grow under the action by a generator~$\sigma_i$, but never go down.}
\label{F:genrlabel}
\end{figure}

\begin{lemma}
If~$[\mathcal C]$ is an~$x$-invariant closed curve of complexity two,
where~$x$ is a positive braid, and the strands of $x$ enclosed by $\mathcal C$ do not cross, then for any prefix~$x'$ of~$x$ the curve $[\mathcal C]^{x'}$ is of complexity two.
\end{lemma}

\begin{proof}
Let $\alpha$ be a descending invariant arc associated to $\mathcal C$. By Lemma~\ref{L:invariant_arc} we know that $\alpha^x =\alpha$. Now, the image of~$\tau_\alpha$ is equal to~$\{0,1\}$, so the same holds for the image of $\tau_{\alpha^x}$. Thus Lemma~\ref{L:maxincreases} implies that for any prefix~$x'$ of~$x$ one has
$$
  1= \max_{t\in I}\,\tau_{\alpha^x}(t) \geqslant \max_{t\in I}\,\tau_{\alpha^{x'}}(t) \geqslant \max_{t\in I}\,\tau_\alpha(t) =1
$$
and
$$
0=\min_{t\in I}\,\tau_{\alpha^x}(t) \geqslant \min_{t\in I}\,\tau_{\alpha^{x'}}(t) \geqslant \min_{t\in I}\,\tau_\alpha(t)=0.
$$
Hence the image of $\tau_{\alpha^{x'}}$ is also equal to~$\{0,1\}$. Therefore $[\mathcal C]^{x'}$ has complexity two.
\end{proof}

Let us introduce some more notation. We shall suppose that~$\alpha$ is
a descending invariant arc of some positive braid~$x$.
We suppose also that~$\alpha'\subset \alpha$ is a sub-arc whose two
extremities lie in two interior punctures.
We say an exterior puncture is \emph{left-blocked} by~$\alpha'$
if there is no smooth path starting at this puncture point, terminating
on the boundary of the disk, disjoint from the arc~$\alpha'$, and whose
tangent direction has always a negative real coordinate.
A \emph{right-blocked} puncture is defined symmetrically.
We define interior punctures to be both left and right blocked.
We shall call the two interior punctures at the two ends of the
arc~$\alpha'$ the \emph{extremal} (interior) punctures of~$\alpha'$.

The proof of Proposition~\ref{P:trivial interior -> almost round} will be
completed by proving that there are no blocked exterior punctures at all,
meaning that the curve~${\mathcal C}$ is of complexity~1. First we obtain
two partial results:

\begin{lemma}\label{L:NoFarOutBlocked} Let $\mathcal C$ be an essential reduction curve for a positive braid $x$. Suppose that~$\alpha'$ is a sub-arc of a descending invariant
arc of $x$. Then there cannot be any exterior punctures which are left-blocked by~$\alpha'$
and to the left of both extremal punctures of~$\alpha'$. Similarly,
there cannot be a right-blocked exterior puncture to the right of
both extremal punctures.
\end{lemma}

\begin{proof} We shall prove the first sentence, the proof of the
second one is very similar. Moreover, we shall suppose that the starting
point of the arc~$\alpha'$ (which in the picture is ``higher''
than the end point) is to the \emph{left} of the terminal point, see
Figure~\ref{F:DescArcEx}(a).
The proof of the other case (where the starting point
of the arc~$\alpha'$ is to the \emph{right} of the end point,
Figure~\ref{F:DescArcEx}(b)) is
similar, one simply has to consider the positive braid
$\mbox{rev}(x)$, which is the image of $x$ under the anti-isomorphism $\mbox{rev}:\: B_n\rightarrow B_n$ which sends $\sigma_i$ to itself for every $i=1,\ldots,n-1$ (that is, $\mbox{rev}(x)$ is equal to $x$ written backwards).

We shall argue by contradiction: let us suppose that there is some
left-blocked puncture which is to the left of the left extremal
interior puncture (see Figure~\ref{F:DescArcEx}(a)). We observe that
the corresponding strands cannot cross in the braid~$x$ -- indeed,
if we think of the braid~$x$ as a dance of the punctures, then
during this dance the left-blocked puncture cannot move \emph{under}
the left extremal interior puncture, for this would require a negative
crossing, and it cannot move \emph{over} it, for this would turn the
curve~$\alpha'$ into a curve~$\alpha''$ which possesses some points where
the function~$t_{\alpha''}$ takes the value~$2$.
Thus the set of punctures which are left-blocked by~$\alpha'$ and which
lie to the left of both endpoints of~$\alpha'$ is stable during the whole
dance.

\begin{figure}[htb]
\centerline{\input{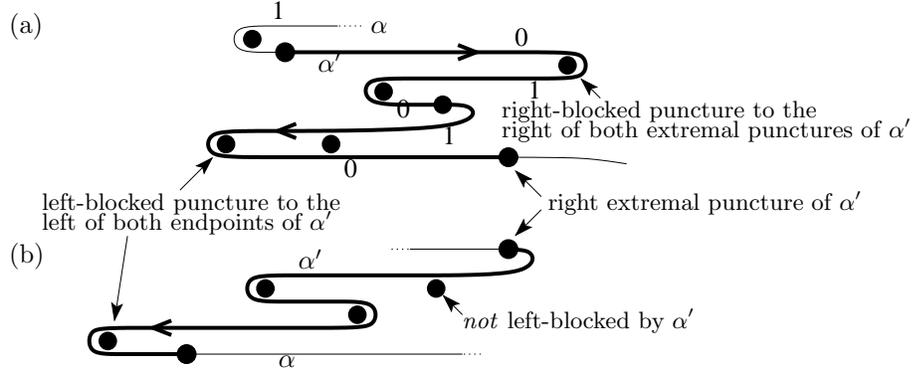}}
\vspace{-2mm}
\caption{(a) The starting point of~$\alpha'$ (the bold line segment)
is to the left of the end point. \ (b) Vice versa.}
\label{F:DescArcEx}
\end{figure}

Now the vertical line through the left extremal interior
puncture, together with the arc~$\alpha'$, cuts the disk into a number
of connected components, at least one of which contains some left-blocked
punctures to the left of the left extremal puncture. Let~$\Xi$ be the
union of all the components containing left-blocked punctures.
Let~$\Psi$ be the union of~$\Xi$ with an initial segment of~$\alpha'$
long enough to touch all the connected components of~$\Xi$, but not
all of~$\alpha'$. (So~$\Psi$ looks in general like some pearls on a
thread, see Figure~\ref{F:ConstrInv}(a).)
Let $N(\Psi)$ be a regular neighbourhood of~$\Psi$. We observe
that $N(\Psi)$ is preserved by the action of~$x$, and so is its
boundary, which we shall call~$\mathcal C'$. Moreover, $\mathcal C'$ intersects the
canonical reduction curve~$\mathcal C$ (which, we recall, was the boundary of a
regular neighbourhood of~$\alpha$) twice. This contradicts the definition
of a canonical reduction curve.
\end{proof}

\begin{figure}[htb]
\centerline{\input{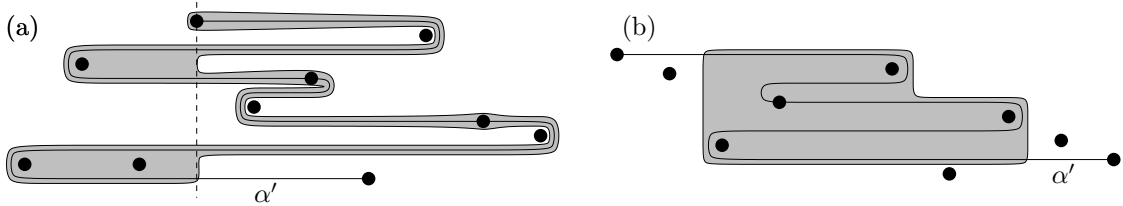}}
\vspace{-2mm}
\caption{Constructing invariant curves which intersect the curve~$c$:
(a) In the case where there is a left blocked puncture to the left of
both extremal punctures, and (b) in the other case.}
\label{F:ConstrInv}
\end{figure}

\begin{lemma}\label{L:NothingBlocked}
Let $\mathcal C$ be an essential reduction curve for a positive braid $x$. Suppose that~$\alpha'$ is a sub-arc of a descending invariant
arc $\alpha$ of $x$. Also suppose that $\alpha'$ does not traverse any interior punctures (except its two endpoints). Then there cannot be any exterior punctures blocked
by~$\alpha'$.
\end{lemma}

\begin{proof} Again, we shall assume that the starting point of~$\alpha'$
is to the \emph{left} of the end point, with the other case being similar.
Lemma~\ref{L:NoFarOutBlocked} together with the hypothesis that $\alpha'$
does not traverse any interior puncture imply that any blocked punctures
would have to lie between the left and the right extremity of~$\alpha'$.
Supposing, for a contradiction, that such
blocked punctures exist, then there must be a pair of them, with a
right-blocked puncture above a left-blocked one (see
Figure~\ref{F:ConstrInv}(b)). Let us now look at the braid~$x$,
considered as a dance of the punctures.

We claim that the two punctures can never cross, and that they stay
between the two extremal punctures at all times. Indeed, a crossing between
the two punctures (while they both lie between the extremal punctures)
would transform the arc~$\alpha'$ into an arc of complexity~$3$.
Moreover, as soon as one of the two punctures quits the region
between the two extremal punctures, it becomes a left-blocked puncture
to the left of both extremal punctures, or a right-blocked puncture
to the right of both extremal punctures, which is impossible by
Lemma~\ref{L:NoFarOutBlocked}. This proves the claim.

Thus any punctures which, at any moment during the puncture dance, are
blocked by~$\alpha'$ and lie between the left and right extremal puncture
of~$\alpha'$, keep these properties throughout the puncture dance. This
helps us to construct an invariant curve in the following manner: we
take the set of all points of~$D^2$ which have points of~$\alpha'$ both
above and below them (see Figure~\ref{F:ConstrInv}(b)).
The punctures contained in this region are precisely those which are left or right blocked by $\alpha'$.  A regular neighbourhood of the closure of this set is a disk, or possibly a
disjoint union of some disks. The boundary of each disk intersects~$\alpha'$
in two points, and hence intersects~$\mathcal C$ in at least two points.
However, the boundary of the disk is $x$-invariant, so we have a
contradiction with the requirement that~$\mathcal C$ belongs to the
canonical reduction system.
\end{proof}

We are now ready to prove the main result in this section.

{\it Proof of Proposition~\ref{P:trivial interior -> almost round}.} (see
Figure~\ref{F:NoCxty2}). We recall that after an isotopy which moves punctures only
vertically, there is an $x$-invariant arc~$\alpha$ which contains all the interior
punctures (i.e.\ punctures inside the canonical reduction
curve~$\mathcal C$), which is monotonically decreasing in height.

\begin{figure}[htb]
\centerline{\input{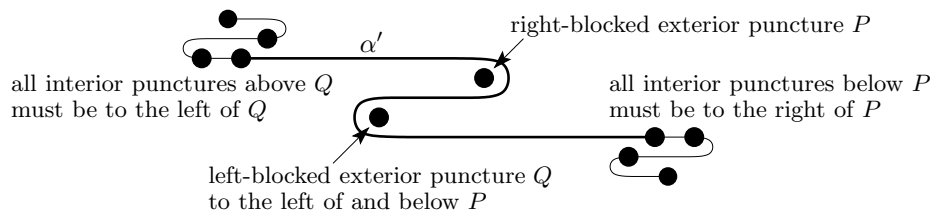}}
\vspace{-2mm}
\caption{The arc $\alpha$ and (bold) its subarc $\alpha'$.}
\label{F:NoCxty2}
\end{figure}

Let us suppose, for a contradiction that there is a blocked exterior
puncture~$P$ -- without loss of generality we suppose it is a
right-blocked one. From Lemma~\ref{L:NoFarOutBlocked} we know that either
all the interior punctures which lie \emph{above}~$P$ must lie to the right
of~$P$, or all the interior punctures \emph{below}~$P$ must lie to the
right of~$P$.
Again without loss of generality we suppose that the latter is the case.

On the other hand, $P$~is right-blocked, so there must be a left-blocked
puncture~$Q$ somewhere below and to the left of~$P$. By the previous
paragraph, this puncture must be exterior. Again by
Lemma~\ref{L:NoFarOutBlocked}, all the interior punctures above~$Q$
must lie to the left of~$Q$.

But now we have an arc~$\alpha'$ starting at one of the interior
punctures above and to the left of~$Q$, ending at one of the interior
punctures below and to the right of~$P$, not traversing any interior
punctures, and yet blocking both~$P$ and~$Q$.
This is in contradiction with Lemma~\ref{L:NothingBlocked}, and
terminates the proof of Proposition~\ref{P:trivial interior -> almost round}.
\hfill$\Box$


\subsection{Detecting reducible braids with trivial interior braids}

\begin{theorem}\label{T:round_almostround}
There is an algorithm which decides whether a given positive braid~$x$ of length~$\ell$ with~$n$ strands preserves an almost-round curve whose interior strands do not cross. Moreover, this algorithm takes time $O(\ell \cdot n^4)$.
\end{theorem}

\begin{proof}
In order to prove this theorem we only have to answer, in time
$O(\ell \cdot n^4)$, the following question:

\emph{Question } Does there exist an embedded arc~$\alpha$ in~$D_n$ which
has its two extremities in two of the punctures, which is almost
horizontal in the sense that the function $\tau_\alpha$ defined in
Section~\ref{SS:CRSAlmostRound} is the constant function
$\tau_\alpha\cong 0$,
and which is invariant under~$x$ (i.e. $\alpha^x\simeq \alpha$)?

In order to answer this question for any given braid~$x$ (with~$n$
strands and of length~$\ell$), we shall think of the braid as a
dance of~$n$ punctures in the disk~$D^2$, where each move of the dance
starts with all punctures lined up on the real line, and consists of an
exchange in a clockwise direction of two adjacent punctures. We shall
often be working with the \emph{closure}~$\wtil x$ of the braid,
and this braid corresponds to a \emph{periodic} dance of the punctures.
Notice that in the braid~$x^n$, every puncture performs at least
one complete cycle, and possibly more, of this periodic dance, in the
sense that for every puncture there exists an integer~$k$
between~$1$ and~$n$ such that~$x^k$ fixes that puncture.

If an almost horizontal, $x$-invariant arc exists, then its
deformed versions remain always almost horizontal during the whole dance,
by Lemma \ref{L:maxincreases} again. Thus a positive answer to the above question
is equivalent to the existence of an almost horizontal arc, at every
point in time, connecting two of the punctures, which varies continuously
with time, and which is invariant under applying one complete period
of the dance. The two endpoints of the arc will be called the
\emph{interior} punctures (because we think of them as being
inside an invariant circle). Notice that these two punctures can
never cross during the dance (because the interior braid is trivial
by hypothesis), so it makes sense to talk about the \emph{left} and the
\emph{right interior puncture}.

If an $x$-invariant, almost horizontal arc exists then, at any moment,
for any puncture lying
between the two endpoints of the arc, we have a well-defined notion
of the puncture lying \emph{above} or \emph{below} the arc.

\begin{lemma}\label{L:ForcedLabels}
Suppose we are given a braid~$x$ together with a choice of two
strands (the ``interior strands'') which are pure and which do not cross
each other. Then an~$x$-invariant, almost horizontal arc connecting
the given interior punctures exists if and only if there
is a way of labelling, at each of the~$\ell$ timesteps, each puncture
by a letter ``$a$'' or ``$b$'' (for ``above'' or ``below''), or to
leave them unlabelled, so that the following restrictions are satisfied:
\begin{enumerate}
\item (a) Every puncture lying between the interior punctures has to
carry a label (``$a$'' or ``$b$'').

(b) Punctures not lying
between the two interior punctures are unlabelled.
\item A puncture preserves its label in~$\wtil x$ until its next
crossings with an interior puncture (and punctures which never cross
any interior puncture have the same label for all time).
\item Rules concerning crossings involving an interior puncture:

(a) If a puncture moves from left to right over the left interior puncture,
then afterwards it must be labelled ``$a$''. Similarly, if a puncture
moves from right to left under the right interior puncture,
then afterwards it must be labelled ``$b$''.

(b) A puncture labelled~``$a$'' may not move from right to left under
the left interior puncture. Similarly, a puncture labelled ``$b$''
may not move from left to right over the right interior puncture.
\item A rule concerning crossings involving two punctures which lie
between the interior punctures: if a puncture labelled~``$a$'' crosses
from right
to left under another puncture, then this puncture must also be
labelled~``$a$''. If a puncture labelled~``$b$'' crosses from left
to right over another puncture, then this puncture must also be
labelled~``$b$''.
\item The invariance rule: the labelling before the first and after
the~$\ell$th timestep is the same.
\end{enumerate}
\end{lemma}

\begin{proof}[Proof of Lemma~\ref{L:ForcedLabels}]
If there is an invariant arc, then we obtain a labelling of the punctures
with the required properties by labelling punctures which lie above the
arc by ``$a$'' and punctures which lie below the arc by ``$b$''.

Conversely, if a labelling of the punctures with the required properties
is given, and if we construct, at time~$t=0$, an almost horizontal arc
going below the punctures labelled~``$a$'' and above the punctures
labelled~``$b$'' at that moment, then acting by any initial segment
of the puncture dance up to some time~$t=T$ yields again an arc which
is almost horizontal and goes above the punctures labelled ``$b$''
and below the punctures labelled ``$a$'' at this moment~$t=T$.
\end{proof}

\begin{lemma}\label{L:SuperflLabels}
The statement of Lemma~\ref{L:ForcedLabels} remains true if
restriction~1(a) is left out, i.e.\ if punctures lying between the
interior punctures at all times may remain unlabelled.
\end{lemma}

\begin{proof}[Proof of Lemma~\ref{L:SuperflLabels}]
Given a labelling satisfying the restrictions of Lemma~\ref{L:ForcedLabels}
except restriction~1(a), we can construct a new labelling simply by
labelling all unlabelled punctures between the two interior
punctures~``$a$''. It is an easy verification that this extended
labelling satisfies the complete list of restrictions.
\end{proof}

Note that there are $\frac{n\cdot (n-1)}{2}$ pairs of punctures, so in
order to prove the theorem, it suffices to construct an algorithm which,
for any given pair of punctures, decides in time $O(\ell \cdot n^2)$
whether there exists a labelling, at each of the~$\ell$ time steps,
of the punctures, respecting the restrictions of
Lemma~\ref{L:SuperflLabels}, with the two given punctures as interior
punctures.


The algorithm is very simple. We start with the $n$-punctured disk,
with two of the punctures designated as the interior punctures but
with all other punctures still unlabelled. We then perform $2n$~times
the puncture dance associated to the braid~$x$.
At each of the~$2n\cdot \ell$ time steps
we perform the labellings of punctures forced by the rules of
Lemma~\ref{L:ForcedLabels}, specifically by rules 1(b), 2, 3(a) and~4.

It may happen that applying these rules leads to a contradiction,
for instance if a puncture already labelled~``$a$'' crosses from right
to left under the left interior puncture (rule~3(b) violated).
A contradiction also arises if a puncture has a label~``$a$'' but the
puncture which sat in the same position $n$~steps previously was
labelled~``$b$'', or vice versa (invariance rule violated).
If such a contradiction occurs, we know that no coherent labelling
exists, so the two chosen punctures cannot be interior punctures of
a reducible braid with trivial interior braid.

If, by contrast, after~$2n$ repetitions of the puncture dance of~$x$
still no contradictions have arisen, then, we claim, we have succeeded
in finding a labelling of~$x$ satisfying the restrictions of
Lemma~\ref{L:ForcedLabels} (except number~1(a)), and this terminates
the description of the algorithm.

Let us now prove the claim we just made. What we need to prove is that
the labelling that we wrote during the $2n$th repetition of the puncture
dance (i.e.\ during the last~$\ell$ steps) satisfies all the conditions
of Lemma~\ref{L:ForcedLabels} (except number~1(a)). In fact, conditions
1(b) through~4 of Lemma~\ref{L:ForcedLabels} are satisfied by
construction, so the only nontrivial claim is that the periodicity
condition holds: if at any
moment during the last~$\ell$ time steps a puncture was labelled~``$a$''
or~``$b$'', then the puncture which was in the same place $\ell$~steps
previously was already labelled, with the same label.
In other words, we have to show that during the~$2n$th repetition of the
puncture dance no label was found that was still unknown during
the~$2n-1$st iteration.

Let us first look at those punctures between the interior punctures
which correspond to strands that have some crossing with the interior
strands. We observe that after the~$n$th repetition of the puncture
dance, these punctures have already experienced a crossing with an
interior puncture, so all these punctures are already labelled
after~$n$ repetitions of the dance. There are at most~$n-2$ punctures
left which never cross the exterior punctures. If during one repetition
of the puncture dance none of these punctures receives a new label,
then none of them ever will in any further iteration of the puncture
dance. Conversely once such a puncture is labelled, it will remain
labelled at all times during all future iterations. Therefore after
another~$n-2$ iterations the labelling process is complete. This terminates
the proof of the last claim, and of Theorem~\ref{T:round_almostround}.
\end{proof}



\section{Sliding circuits}\label{S:sliding_circuits}

In this section we will introduce the special type of conjugations called \emph{cyclic slidings} and state our main result, Theorem~\ref{T:main}. Recall that the braid group~$B_n$ admits a lattice order~$\preccurlyeq$, called the \emph{prefix order}, defined as follows: we say that $a\preccurlyeq b$ if and only if~$ac=b$ for some positive element $c\in B_n$. Here \emph{positive} means that it can be written as a product of positive powers of the standard generators $\sigma_1,\ldots,\sigma_{n-1}$. Being a lattice order,~$\preccurlyeq$ determines a unique greatest common divisor~$x\wedge y$ and a unique least common multiple~$x\vee y$ of every pair of braids $x,y\in B_n$.   There is also another lattice order in~$B_n$, denoted~$\succcurlyeq$, and called the \emph{suffix order}. This time we say that~$a\succcurlyeq b$ if~$a=cb$ for some positive braid~$c$. The gcd and lcm determined by~$\succcurlyeq$ on a pair of elements $x,y\in B_n$ will be denoted $x\wedge_R y$ and $x\vee_R y$, respectively. Due to these lattice structures, $B_n$~is the main example of a Garside group~\cite{DP}.

Recall that each element $x\in B_n$ admits a \emph{left normal form} (or left greedy normal form)~\cite{Epstein}, that is, a unique way to decompose it as $x=\Delta^p x_1\cdots x_r$, where~$p\in\Z$ is maximal and each~$x_i$ is a proper simple braid (a permutation braid different from~$1$ and~$\Delta$) such that $x_ix_{i+1}$ is left-weighted. This latter property means that for every $k=1,\ldots,n-1$, if~$x_{i+1}$ can be written as a positive word starting by~$\sigma_k$, then~$x_i$ can be written as a positive word ending by~$\sigma_k$.  This can also be described in terms of the lattice properties of~$B_n$: for a simple element~$s$, define $\partial(s)=s^{-1}\Delta$, the \emph{ complement} of~$s$. Then $x_ix_{i+1}$ being left-weighted means that $\partial(x_i)\wedge x_{i+1}=1$. We recall that if the left normal form of~$x$ is as above, then the integers $p$,~$r$ and $p+r$ are called \emph{infimum}, \emph{canonical length} and \emph{supremum} of~$x$, respectively, and they are denoted $\inf(x)$, $\ell(x)$ and~$\sup(x)$.

As usual, for every $x\in B_n$ we define $\tau(x)=\Delta^{-1}x\Delta$. We also define the \emph{initial factor} of~$x$ to be $\iota(x)=\Delta \wedge (x\Delta^{-\inf(x)})$. That is, if the left normal form of~$x$ is $\Delta^p x_1\cdots x_r$ and $r>0$, one has $\iota(x)=\tau^{-p}(x_1)$, and if~$r=0$ one has $\iota(x)=1$.  We also define the \emph{final factor} of~$x$ to be $\varphi(x)=\left( x\wedge \Delta^{\sup(x)-1}\right)^{-1}x$. That is, if the left normal form of~$x$ is $\Delta^p x_1\cdots x_r$ and $r>0$, one has $\varphi(x)=x_r$, and if $r=0$ one has $\varphi(x)=\Delta$. It is well known~\cite{BGG1} that for every $x\in B_n$, one has $\varphi(x)\iota(x^{-1})=\Delta$. In other words, $\iota(x^{-1})=\partial(\varphi(x))$.

\begin{definition}\label{D:cyclic_sliding}{\rm \cite{GG1}} Given $x\in B_n$, its \emph{preferred prefix} is the simple element
 $$
   \mathfrak p(x)=\iota(x)\wedge \iota(x^{-1})= \iota(x) \wedge \partial(\varphi(x)).
 $$
\emph{Cyclic sliding} of~$x$ means the conjugation of~$x$ by its preferred prefix, and its result is denoted $\mathfrak s(x)$. That is, $\mathfrak s(x)=\mathfrak{p}(x)^{-1} x\: \mathfrak{p}(x)$.
\end{definition}

Notice that for any braid~$x$ we have $\mathfrak p(x^{-1})=\mathfrak p(x)$, so $\mathfrak s(x^{-1})=(\mathfrak s(x))^{-1}$. It is explained in~\cite{GG1} why the above definition is natural, and that cyclic sliding is a conjugation which simplifies any given braid from the algebraic point of view. For instance, iterated application of cyclic sliding sends any braid~$x$ to a conjugate~$\widetilde x$ which belongs to its ultra summit set~\cite{Gebhardt}. In particular, the canonical length of~$\widetilde x$ is minimal in its conjugacy class. Iterated application of cyclic sliding to an element~$x$ always yields a repetition, so the orbit of~$x$ under~$\mathfrak s$ becomes periodic. We call the set of elements in that periodic orbit the {\it sliding circuit} associated to~$x$. The union of all sliding circuits in the conjugacy class of~$x$, is thus the set of elements that cannot be improved in any sense by further application of~$\mathfrak s$. We call it the {\it set of sliding circuits} of~$x$. More precisely:

\begin{definition}{\rm \cite{GG1}}
Given $x\in B_n$, let $x^{B_n}$ be its conjugacy class. We define the \emph{set of sliding circuits} of~$x$ as follows:
$$
    SC(x)=\{y\in x^{B_n};\ \mathfrak s^m(y)=y \mbox{ for some } m>0\}.
$$
\end{definition}

The main goal of this paper is to show that cyclic sliding also simplifies braids from the geometrical point of view. But in order to achieve this, we need to study the elements together with some of their powers. The main problem is that if~$y$ belongs to a sliding circuit, $y^2$ does not necessarily have the same property, and neither does~$y^m$ for arbitrary~$m$.  In order to consider elements which behave nicely with respect to cyclic sliding and (at least some) powers, we introduce the following notion.

\begin{definition}
Given $x\in B_n$ and an integer~$m> 0$, we define the \emph{$m$th stabilized
set of sliding circuits}
$$
   SC^{[m]}(x)= \{y\in B_n;\ y^k\in SC(x^k) \mbox{ for } k=1,\ldots,m\}.
$$
\end{definition}

Notice that the elements of $SC^{[m]}(x)$ are precisely those that are conjugate to~$x$ and with the property that their first~$m$ powers each belong to their own sliding circuit. The main result in this paper is
the following theorem.

\begin{theorem}\label{T:main}
Let $x\in B_n$ be a non-periodic, reducible braid. There is some $m\leqslant ||\Delta||^3-||\Delta||^2$ such that every element $y\in SC^{[m]}(x)$ admits an essential reduction curve which is either round or almost round.
\end{theorem}


We will show in Section~\ref{S:SC and powers} that $SC^{[m]}(x)$ is nonempty for every $m>0$, and we will give an algorithm for finding an element in $SC^{[m]}(x)$. Moreover, if~$m$ is bounded by a polynomial in~$n$ and~$\ell(x)$, then the complexity of the algorithm will be polynomial in~$n$ and~$\ell(x)$, provided the following well-known conjecture is true:

\begin{conjecture}\label{C:no of slidings}{\rm \cite{GG1}}
Given $x\in B_n$ of canonical length~$\ell$, let~$t$ be the minimal positive integer such that
$\mathfrak s^k(x)=\mathfrak s^t(x)$ for some~$k$ with $0\leqslant k<t$.
Then~$t$ is bounded by a polynomial in~$\ell$ and~$n$.
\end{conjecture}

Together with Theorem~\ref{T:round_almostround} this yields an algorithm (Algorithm 1 in the Introduction) to determine whether a given element of~$B_n$ is periodic, reducible or pseudo-Anosov. This algorithm will be polynomial in~$n$ and~$\ell$ if Conjecture~\ref{C:no of slidings} is true.


\section{Sliding circuits and powers}\label{S:SC and powers}

This section is devoted to the study of the set $SC^{[m]}(x)$ defined
in the previous section. We will show that for every $x\in B_n$ (actually, in every Garside group), and for every~$m>0$, the set $SC^{[m]}(x)$ is nonempty. Furthermore, we shall give a simple procedure to compute one element in $SC^{[m]}(x)$, starting from~$x$.

In order to achieve these goals, we shall need the following result:

\begin{theorem}\label{T:convexity}{\rm \cite{GG1}}
Let $x,z,a,b\in B_n$.  If $z,z^a,z^b\in SC(x)$, then $z^{a\wedge b}\in SC(x)$.
\end{theorem}


\begin{definition} 
Given $y\in B_n$, we define the \emph{preferred conjugator} $P(y)$ of~$y$ as the product of conjugating elements corresponding to iterated cyclic sliding until the first repetition. That is, if we denote $y^{(i)}=\mathfrak s^i(y)$ for $i\geqslant 0$ and if~$t$ is the smallest positive integer such that $\mathfrak s^t(y)=\mathfrak s^i(y)$ for some~$i<t$, then:
$$
    P(y)= \mathfrak p(y)\; \mathfrak p(y^{(1)}) \; \mathfrak p(y^{(2)}) \cdots \mathfrak p(y^{(t-1)}).
$$
\end{definition}

Notice that if one conjugates~$y$ by~$P(y)$, one obtains an element in~$SC(y)$. Notice also that if $y\in SC(x)$ for some~$x$, then~$P(y)$ is the conjugating element along the whole sliding circuit of~$y$. In particular, if $y\in SC(x)$ then~$P(y)$ commutes with~$y$.

\begin{definition}
Given $x\in B_n$, we define $x_{[0]}=x$, and for every~$m\geqslant 0$ we define recursively:
$$
       x_{[m+1]}= (x_{[m]})^{P((x_{[m]})^{m+1})}.
$$
\end{definition}

Notice that~$x_{[i]}$ is a conjugate of~$x$ for every~$i>0$, since in order to compute $x_{[m+1]}$ we are just conjugating~$x_{[m]}$. The conjugating element is precisely the one that sends the~$m+1$st power of~$x_{[m]}$ to a sliding circuit.

The proof of the following result gives a simple algorithm to compute one element in $SC^{[m]}$. This parallels Proposition 2.23 in~\cite{BGG1}.

\begin{proposition}
Let $x\in B_n$ and~$m>0$. Then $x_{[m]}\in SC^{[m]}(x)$. In particular, $SC^{[m]}(x)$ is nonempty.
\end{proposition}

\begin{proof}
We will show the result by induction on~$m$.  For~$m=1$, one has $x_{[1]}=(x_{[0]})^{P(x_{[0]})}= x^{P(x)}$, which by definition of~$P(x)$ belongs to $SC(x)=SC^{[1]}(x)$.

Now assume that $x_{[m]}\in SC^{[m]}(x)$ for some~$m>0$. That is, $(x_{[m]})^i\in SC(x^i)$ for $i=1,\ldots,m$. The~$m+1$st power $(x_{[m]})^{m+1}$ does not, a priori, belong to a sliding circuit. But if we conjugate our braid~$x_{[m]}$ by the element $P((x_{[m]})^{m+1})$ to obtain $x_{[m+1]}$, its~$m+1$st power becomes the conjugate of $(x_{[m]})^{m+1}$ by $P((x_{[m]})^{m+1})$, so it belongs to a sliding circuit as desired.  The question to be answered is whether smaller powers of $x_{[m+1]}$ still belong to a sliding circuit or not. That is, we have to show that for $i=1,\ldots,m$, the conjugate of $(x_{[m]})^i$ by $P((x_{[m]})^{m+1})$ belongs to $SC(x^i)$.

First we claim that if~$y$ is a braid such that $y^i\in SC(x^i)$ for $i=1,\ldots,m$, then the conjugate of~$y^i$ by $\mathfrak p(y^{m+1})$ also belongs to $SC(x^i)$, for $i=1,\ldots,m$. In order to prove this claim, we recall from Theorem~\ref{T:convexity} that if $z,a,b\in B_n$ are braids such that~$z$, $z^a$ and $z^b$ belong to a sliding circuit, then $z^{a\wedge b}$ also belongs to a sliding circuit. In the situation of the claim, we have $y^i\in SC(x^i)$, and $\mathfrak p(y^{m+1})= \iota(y^{m+1})\wedge \iota(y^{-m-1})$. Recall that $\iota(y^{m+1})=\left(y^{m+1}\Delta^{s}\right) \wedge \Delta$ for some integer~$s$. Since the conjugate of~$y^i$ by $y^{m+1}\Delta^{s}$ is $\tau^s(y^i)\in SC(x^i)$, and the conjugate of~$y^i$ by~$\Delta$ is $\tau(y^i)\in SC(x^i)$, Theorem~\ref{T:convexity} shows that the conjugate of~$y^i$ by $\iota(y^{m+1})$ belongs to a sliding circuit. The same argument shows that its conjugate by  $\iota(y^{-m-1})$ also belongs to a sliding circuit. Hence, applying Theorem~\ref{T:convexity} again, the claim is shown.

As an aside, we remark that conjugation by $\mathfrak p(y^{m+1})$ performs a cyclic sliding on~$y^{m+1}$, but it does not perform a cyclic sliding on~$y^i$ for $i<m+1$. Nevertheless, we just showed that the conjugate of~$y^i$ by $\mathfrak p(y^{m+1})$, even if it is not necessarily the cyclic sliding of~$y^i$, belongs to a sliding circuit (which is not necessarily the same sliding circuit~$y^i$ belonged to).

Now, in our situation, $P((x_{[m]})^{m+1})$ is the product of several preferred prefixes, those of iterated cyclic slidings of $(x_{[m]})^{m+1}$, so conjugation by $P((x_{[m]})^{m+1})$ is the composition of several conjugations, by $\alpha_1, \ldots, \alpha_t$, say. We can then apply the above claim several times, taking $y=x_{[m]}^{\alpha_1\cdots \alpha_{k-1}}$ for $k=1,\ldots,t$.  At the first step $y=x_{[m]}$, so $y^i\in SC(x^i)$ for $i=1,\ldots,m$, and $\alpha_1=\mathfrak p((x_{[m]})^{m+1})=\mathfrak p(y^{m+1})$ by definition. By the claim, $(y^{\alpha_1})^i$ belongs to its own sliding circuit for $i=1,\ldots, m$. By induction, if $y=x_{[m]}^{\alpha_1\cdots \alpha_{k-1}}$ for some~$k\geqslant 1$, and~$y^i$ belongs to its own sliding circuit for $i=1,\ldots,m$, then the conjugate of~$y^i$ by~$\alpha_k$ also belong to its own sliding circuit, as $\alpha_k=\mathfrak p(y^{m+1})$. For~$k=t$, as $x_{[m]}^{\alpha_1\cdots \alpha_k}= x_{[m+1]}$, this means that the first~$m$ powers of~$x_{[m+1]}$ belong to their own sliding circuit. Since the~$m+1$st power also belongs to its sliding circuit by construction, the result follows.
\end{proof}

Recall Conjecture~\ref{C:no of slidings} above. Let $T=T_{n,\ell}$ be an upper bound for the number of cyclic slidings necessary to obtain a repetition, starting from a braid in~$B_n$ of canonical length~$\ell$. This bound~$T_{n,\ell}$ is thus conjectured to be a polynomial in~$n$ and~$\ell$. We also recall from~\cite{Epstein} that if~$x$ and~$y$ are two braids, given in left normal form, of canonical length~$\ell_1$ and~$\ell_2$, respectively, the left normal form of~$xy$ can be computed in time $O(\ell_1\ell_2n\log n)$. This comes from the fact that the gcd of two simple elements can be computed in time $O(n\log n)$~\cite{Epstein}, and from the way in which left normal forms are computed.

\begin{corollary}\label{C:computing_one_element}
Given $x\in B_n$ written as a product of~$\ell$ simple elements and its inverses, and given~$m>0$, there is an algorithm that computes an element in $SC^{[m]}(x)$ in time $O(S\ell n\log n)$, where $\displaystyle S=\sum_{i=1}^m{i\:T_{n,i\ell}}$.
\end{corollary}

\begin{proof}
The algorithm computes $P(x_{[i-1]}^{i})$ and conjugates~$x_{[i-1]}$ by this element (obtaining $x_{[i]}$), for $i=1,\ldots,m$.

We start with~$x$ written as a product of~$\ell$ simple elements and its inverses, and compute its left normal form, which takes time $O(\ell^2n\log n)$~\cite{Epstein}. Now we apply iterated cyclic sliding to~$x$ until the first repetition, which is~$x_{[1]}$. At each step, we have to compute the preferred prefix of an element~$\alpha$, and conjugate~$\alpha$ by it.  Notice that a preferred prefix is the greatest common divisor of two permutation braids: if $\alpha=\Delta^p \alpha_1\cdots \alpha_r$ is in left normal form and $r>0$, then $\mathfrak p(\alpha)=\tau^{-p}(\alpha_1)\wedge \partial(\alpha_r)$. If the left normal form of~$\alpha$ is known, the computation of $\tau^{-p}(\alpha_1)$ and $\partial(\alpha_r)$ takes time~$O(n)$~\cite{Epstein}, and computing their gcd takes time $O(n\log n)$~\cite{Epstein}. Now~$\alpha$ is an iterated cyclic sliding of~$x$, where cyclic sliding never increases the canonical length of an element~\cite{GG1}. Hence the canonical length of~$\alpha$ is at most~$\ell$. The algorithm takes~$\alpha$ in left normal form, and computes the left normal form of its conjugate by $\mathfrak p(\alpha)$. As~$\mathfrak p(\alpha)$ is a simple element, and~$\alpha$ has canonical length at most~$\ell$, this last step takes time $O(\ell n\log n)$~\cite{Epstein}. 
Thus computing~$\mathfrak p(\alpha)$, conjugating~$\alpha$ by it, and calculating the left normal form of the result takes time $O(\ell n\log n)$. This is repeated $T_{n,\ell}$ times, so~$x_{[1]}$ is computed in time $O(T_{n,\ell}\,\ell n\log n)$.

In the following steps of the algorithm, one has $x_{[i-1]}$ and $x_{[i-1]}^{i-1}$ written in left normal form (the case of the previous paragraph is~$i=1$). Notice that the canonical length of $x_{[i-1]}^{i-1}$ is at most $(i-1)\ell$. The algorithm then computes the left normal form of $x_{[i-1]}^{i}$. This computation, obtained from the product of the left normal forms of $x_{[i-1]}$ and $x_{[i-1]}^{i-1}$, takes time $O((i-1)\ell^2 n\log n)$. Now the algorithm computes iterated cyclic slidings of $x_{[i-1]}^{i}$ until the first repetition. More precisely, the algorithm starts with $\alpha=x_{[i-1]}$, and at each step it computes the preferred prefix $\mathfrak p(\alpha^{i})$, and conjugates both~$\alpha$ and~$\alpha^{i}$ by this prefix. The conjugate of~$\alpha$ is set as the new value of~$\alpha$, and the loop is repeated. The loop ends at the first repetition of~$\alpha^{i}$. The complexity of this computation is the same as that of the previous paragraph, but applied to a braid of canonical length~$i\ell$, instead of~$\ell$. Hence, the computation of~$x_{[i]}$ and $x_{[i]}^{i}$ from $x_{[i-1]}$ and $x_{[i-1]}^{i-1}$ takes time $O(T_{n,i\ell}\:i\ell n\log n)$. Adding up the complexities of each loop, we obtain that the whole algorithm takes time $O(\ell^2 n\: \log n)+O(S\ell n\log n)$. As $\ell < T_{n,\ell}\leqslant S$, the result follows.
\end{proof}

We remark that if Conjecture~\ref{C:no of slidings} holds, that is, if $T_{n,\ell}$ is a polynomial in~$n$ and~$\ell$, then the complexity of the algorithm in Corollary~\ref{C:computing_one_element} is polynomial in $n$, $\ell$ and~$m$. As we shall only need to compute one element in $SC^{[m]}(x)$ for $m\leqslant ||\Delta||^3=n^3(n-1)^3/8$ (see Theorem~\ref{T:main}), the complexity in this case will be polynomial in~$n$ and~$\ell$, always provided Conjecture~\ref{C:no of slidings} holds.


\section{Sliding circuits and reduction curves}\label{S:SC and curves}


\subsection{Sliding circuits and round curves}

In this section we shall investigate the properties of the elements belonging
to $SC^{[m]}(x)$, with respect to their canonical reduction systems. The simplest case occurs when this reduction system is made of round curves. The following result assures the existence of these examples

\begin{theorem}\label{T:BGN}{\rm \cite{BGN} (see also~\cite{LL})}
Let $x\in B_n$ be a positive braid whose left normal form is $x_1\cdots x_r$. If~$[\mathcal C]$ is a round curve such that $[\mathcal C]^x$ is also round, then $[\mathcal C]^{x_1\cdots x_i}$ is round for $i=1,\ldots,r$.
\end{theorem}

In other words, if the roundness of a curve is preserved by a braid~$x$, then it is preserved by each factor in the left normal form of~$x$.  Since $\Delta^{\pm 1}$ preserves the roundness of every curve, the above result can be applied to every braid, not necessarily positive.  This is used in~\cite{BGN} to show that, if a braid preserves a round curve, its cycling and its decycling also preserve round curves. This immediately implies that for every reducible braid~$x$, there is some element in its super summit set $SSS(x)$ which preserves a round curve~\cite{BGN}. Clearly, one can replace $SSS(x)$ by $USS(x)$ in the previous statement. Even better, one can replace it by $SC(x)$, as we will now see, but the proof of this fact is slightly different: we need to show the following result, concerning invariant families of round curves.

\begin{proposition}\label{P:sliding_preserves_roundness}
Let $x\in B_n$, and let~$\mathcal F$ be a family round curves such that $[\mathcal F]^x=[\mathcal F]$. Then $[\mathcal F]^{\mathfrak p(x)}$ is also a family of round curves.   Hence, if~$x$ preserves a family of round curves, then so does~$\mathfrak s(x)$.
\end{proposition}

\begin{proof}
We can assume~$r>0$. Let $\Delta^p x_1\cdots x_r$ be the left normal form of~$x$. By Theorem~\ref{T:BGN} applied to each particular curve of~$\mathcal F$, one has that $[\mathcal F]^{\Delta^p x_1}$ is a family of round curves, and since $\Delta^p x_1=\tau^{-p}(x_1)\Delta^p$, it follows that the curves of $[\mathcal F]^{\tau^{-p}(x_1)}$ are round. Let~$\mathcal F_2$ be a family of curves such that $[\mathcal F_2]=[\mathcal F]^{\tau^{-p}(x_1)}$. In the same way, Theorem~\ref{T:BGN} tells us that the curves of $[\mathcal F]^{\Delta^p x_1\cdots x_{r-1}}$ are round. Let~$\mathcal F_1$ be such that $[\mathcal F_1]=[\mathcal F]^{\Delta^p x_1\cdots x_{r-1}}$. Notice that $[\mathcal F_1]^{x_r} =[\mathcal F]^{(\Delta^p x_1\cdots x_{r-1}) x_r} = [\mathcal F]^x = [\mathcal F]$.

We then have $[\mathcal F_1]^{x_r \tau^{-p}(x_1)}=[\mathcal F_2]$, where~$\mathcal F_1$ and~$\mathcal F_2$ are families of round curves. Now, by definition, the left normal form of $x_r \tau^{-p}(x_1)$ is equal to $y_1y_2$, where $y_1=x_r\mathfrak p(x)$.  By Theorem~\ref{T:BGN} again, we obtain that the curves of $[\mathcal F_1]^{y_1}$ are round. But $[\mathcal F_1]^{y_1} = [\mathcal F_1]^{x_r \mathfrak p(x)}= [\mathcal F]^{\mathfrak p(x)}$, hence $[\mathcal F]^{\mathfrak p(x)}$ is a family of round curves, as we wanted to show.
\end{proof}

\begin{corollary}\label{C:round curves in SC^m}
For every reducible braid $x\in B_n$ and every~$m>0$, there is some $y\in SC^{[m]}(x)$ such that $CRS(y)$ consists of round curves. Moreover, all elements in the sliding circuit of~$y$ satisfy the same property.
\end{corollary}

\begin{proof}
The canonical reduction system $CRS(x)$ is a family of disjoint simple curves on the punctured disc. Hence some orientable automorphism of the punctured disc relative to the boundary, will send it to a collection of (possibly nested) round curves.  This automorphism corresponds to a braid $\gamma\in B_n$. In other words, there is some $\gamma\in B_n$ such that $[CRS(x)]^\gamma$ consists of round curves. It is well known that $[CRS(x)]^\gamma = [CRS(x^\gamma)]$, hence $z=x^\gamma$ is a conjugate of~$x$ whose canonical reduction system consists of round curves.

Now recall that $z_{[m]}$, which is the conjugate of~$z$ by $P(z)P((z_{[1]})^2)P((z_{[2]})^3)\cdots P((z_{[m-1]})^m)$, belongs to $SC^{[m]}(z)=SC^{[m]}(x)$. We will show that all the curves in $CRS(z_{[m]})$ are round circles by induction on~$m$. We know that this is true for~$m=0$ since $z_{[0]}=z$, so we assume $CRS(z_{[m-1]})$ consists of round curves for some~$m> 0$.

In order to compute $z_{[m]}$, we conjugate $z_{[m-1]}$ by $P((z_{[m-1]})^{m})$. Recall that the canonical reduction system of an element coincides with the canonical reduction system of each nonzero power, hence $CRS((z_{[m-1]})^{m})$ consists of round curves. Applying iterated cyclic sliding to $(z_{[m-1]})^{m}$ until the first repetition, that is, conjugating it by $P((z_{[m-1]})^{m})$, one obtains $(z_{[m]})^{m}$. By Proposition~\ref{P:sliding_preserves_roundness}, all curves in $CRS((z_{[m-1]})^{m})$ keep their roundness after each application of~$\mathfrak s$. Hence all curves in $CRS((z_{[m]})^{m})=CRS(z_{[m]})$ are round, as we wanted to show.

We have then shown that there is some $y\in SC^{[m]}(x)$ all of whose reduction curves are round. By Proposition~\ref{P:sliding_preserves_roundness} again, the same happens for every element obtained by applying iterated cyclic sliding to~$y$, that is, for every element in the sliding circuit of~$y$.
 \end{proof}

Notice that the above proof does not provide an algorithm to find~$y$, since we do not know a priori which is the braid~$\gamma$ that conjugates~$x$ to~$z$. Nevertheless, since $SC^{[m]}(x)$ is a finite set, one can compute the whole $SC^{[m]}(x)$ and check for each element whether it preserves some family of round curves. In this way one can find a reduction curve for~$y$, and then for~$x$.

The computation of the whole set $SC^{[m]}(x)$, starting from a single element, parallels the usual constructions given in~\cite{EM,BGG1,GG1}, so we will skip it here. For our purposes, it suffices to know that there is one element~$y$ in $SC^{[m]}(x)$ all of whose essential curves are round. Such elements have a particularly nice behavior with respect to normal forms, as it is shown in~\cite{LL} and~\cite{G2009}.

\begin{lemma}\label{L:lnf_round_curves}{\rm (see for instance~\cite{LL})}
Let~$y\in B_n$, and let~$\mathcal F$ be a family of round curves such that $\mathcal F^y$ is also round. Suppose that~$y$ is a positive braid, and let $y_1\cdots y_r$ be its left normal form, where some of the initial factors may be equal to~$\Delta$. Let $\mathcal C\in \mathcal F\cup \partial(D)$. For $i=1,\ldots,r$, denote~$[\mathcal C_i]=[\mathcal C]^{y_1\cdots y_{i-1}}$ and $[\mathcal F_i]=[\mathcal F]^{y_1\cdots y_{i-1}}$. Then the left normal form of $y_{[\mathcal C\in \mathcal F]}$ is precisely ${y_1}_{[\mathcal C_1\in \mathcal F_1]}\: {y_2}_{[\mathcal C_2\in \mathcal F_2]} \cdots {y_r}_{[\mathcal C_r\in \mathcal F_r]}$.   In this normal form, some of the initial factors could be half twists, and some of the final factors could be trivial.
\end{lemma}

\begin{lemma}\label{L:gcd_round_curves}
Let $x,y\in B_n$ be braids, let~$\mathcal F$ be a family of round curves, and let  $[\mathcal C]\in [\mathcal F]\cup \partial(D)$. Suppose that~$\mathcal F^x$ and $\mathcal F^y$ are round. Then $\mathcal F^{x\wedge y}$ is also round, and $(x\wedge y)_{[\mathcal C \in \mathcal F]}= x_{[\mathcal C\in \mathcal F]}\wedge y_{[\mathcal C\in \mathcal F]}$.
\end{lemma}

\begin{proof}
The first sentence is shown be Lee and Lee~\cite{LL}, and the second one in~\cite{G2009}.
\end{proof}

\begin{lemma}\label{L:iota_round_curves}{\rm \cite{G2009}}
Let $y\in B_n$, and let~$\mathcal F$ be a family of round curves such that $\mathcal F^y$ is also round. Let $\mathcal C\in \mathcal F\cup\partial (D)$. Then~$\iota(y)$ preserves the roundness of~$[\mathcal F]$, and $\iota(y)_{[\mathcal C\in \mathcal F]}$ is either a half twist or equal to $\iota(y_{[\mathcal C\in \mathcal F]})$.
\end{lemma}

\begin{proposition}\label{P:preferred_prefix_round_curves}{\rm \cite{G2009}}
Let $y\in B_n$, and let~$\mathcal F$ be a family of round curves such that $[\mathcal F]^y=[\mathcal F]$. Consider the preferred prefix~$\mathfrak p(y)$, and let $\mathcal C\in \mathcal F\cup\partial (D)$. Then $\mathfrak p(y)_{[\mathcal C\in \mathcal F]}$ is either a half twist, or equal to $\mathfrak p(y_{[\mathcal C\in \mathcal F]})$, or to $\iota(y_{[\mathcal C\in \mathcal F]})$, or to $\iota(y_{[\mathcal C\in \mathcal F]}^{-1})$.
\end{proposition}


\subsection{Rigidity, sliding circuits and preferred conjugators}

The key ingredient for showing the main theorem will be the properties of the preferred conjugator~$P(y)$ of a braid~$y$ which preserves a family of round curves. In fact, we won't be able to gain sufficient control over~$P(y)$, and we have to study the preferred conjugator~$P(y^k)$ for some suitable power~$y^k$ of~$y$ instead. The need of taking powers to obtain a better behavior of the preferred conjugator is the reason why we have to work with the set $SC^{[m]}(x)$, rather than simply the set of sliding circuits~$SC(x)$.

The property we will require for a power of $y\in SC(x)$ involves the notion of \emph{rigidity} introduced in~\cite{BGG1}, which measures how the left normal form of an element varies when taking its square. More precisely, if $x=\Delta^p x_1\cdots x_r$ is in left normal form with~$r>0$, one could expect that the left normal form of~$x^2$ is $\Delta^{2p}\tau^p(x_1)\cdots \tau^p(x_r) x_1 \cdots x_r$, but in general this is not the case. We say that the rigidity of~$x$ is  $\mathcal R(x)=k/r$ if~$k$ is the biggest integer in $\{0,1,\ldots,r\}$ such that the first $2|p|+k$ factors in the left normal form of~$x^2$ are $\Delta^{2p}\tau^{p}(x_1)\cdots \tau^p(x_k)$. The two extreme cases are $\mathcal R(x)=0$, in which all factors in the left normal form of~$x$ are modified when considering~$x^2$, and $\mathcal R(x)=1$, in which no factor is modified, and the left normal form of~$x^2$ is the expected one we saw above. In this latter case we say that~$x$ is {\it rigid}.

We will be interested in the case in which $\mathcal R(x)>0$ and $\mathcal R(x^{-1})>0$. This kind of elements are characterized by the following result.

\begin{lemma}\label{L:rigidity_conditions}{\rm \cite[Lemmas 3.4, 3.5 and Corollary 3.6]{BGG1}}
Let $x\in B_n$ with $\ell(x)>0$. The following conditions are equivalent:\vspace{-3mm}
\begin{enumerate}
\item $\mathcal R(x)>0$.\vspace{-2pt}

\item $\inf(x^2)=2\inf(x)$ and $\iota(x^2)=\iota(x)$.\vspace{-2pt}

\item $\inf(x^m)=m\inf(x)$ and $\iota(x^m)=\iota(x)$ for every~$m>0$.\vspace{-2pt}

\end{enumerate}

The following conditions are also equivalent:\vspace{-3mm}

\begin{enumerate}

\item $\mathcal R(x^{-1})>0$.\vspace{-2pt}

\item $\sup(x^2)=2 \sup(x)$ and $\varphi(x^2)=\varphi(x)$.\vspace{-2pt}

\item $\sup(x^m)=m \sup(x)$ and $\varphi(x^m)=\varphi(x)$ for every~$m>0$.

\end{enumerate}
\end{lemma}

These equalities of infima, suprema, initial and final factors yield a good behavior of the preferred conjugators, as we shall see.  Moreover, this condition is preserved by cyclic sliding, if the element is in its super summit set:

\begin{lemma}\label{L:rigidity_and_sliding}
Let~$x\in B_n$ and $y\in SSS(x)$ with~$\ell(y)>0$. Then $\mathcal R(\mathfrak s(y))\geqslant \mathcal R(y)$ and $\mathcal R(\mathfrak s(y)^{-1})\geqslant \mathcal R(y^{-1})$.
\end{lemma}

\begin{proof}
Let $r=\ell(y)>0$. Since $y\in SSS(x)$ one has $\mathfrak s(y)\in SSS(x)$, hence $\ell(\mathfrak s(y))=r$. Notice that the property $\mathcal R(y)\geqslant k/r$ can be rewritten as $y^2\wedge \Delta^{2p+k} = (y \wedge \Delta^{p+k})\Delta^p$.  One can apply to this equality the transport map based at $y$~\cite{GG1}. This map sends~$y$ to~$\mathfrak s(y)$, $\Delta$~to itself, and preserves products and greatest common divisors. Hence one obtains $\mathfrak s(y)^2 \wedge \Delta^{2p+k} = (\mathfrak s(y) \wedge \Delta^{p+k})\Delta^p$, which is equivalent to $\mathcal R(\mathfrak s(y))\geqslant k/r$.  Hence $\mathcal R(y)\geqslant k/r$ implies $\mathcal R(\mathfrak s(y))\geqslant k/r$ for every $k\in \{0,\ldots,r\}$, so one has $\mathcal R(\mathfrak s(y))\geqslant \mathcal R(y)$.

Replacing~$y$ by~$y^{-1}$, which is also in its super summit set, one has $\mathcal R(\mathfrak s(y^{-1}))\geqslant \mathcal R(y^{-1})$. The result follows as $\mathfrak s(y^{-1})=\mathfrak s(y)^{-1}$  (see the argument that follows Definition~\ref{D:cyclic_sliding}).
\end{proof}

The elements in a sliding circuit that fulfill the required rigidity conditions also satisfy the following important property: their preferred conjugator is rigid.

\begin{proposition}\label{P:preferred prefix is rigid}
Let $x\in B_n$ and $y\in SC(x)$ with $\ell(y)>0$. If $\mathcal R(y)>0$ and $\mathcal R(y^{-1})>0$, then the product $\mathfrak p(y)\mathfrak p(\mathfrak s(y))$ is left-weighted, and~$P(y)$ is rigid.
\end{proposition}

\begin{proof}
Let us first prove that $\mathfrak p(y) \: \mathfrak p(\mathfrak s(y))$ is left-weighted. Consider the biggest element $\alpha \preccurlyeq \mathfrak p(\mathfrak s(y))$ such that $\mathfrak p(y)\alpha$ is simple. Let $\Delta^p y_1\cdots y_r$ be the left normal form of~$y$. Notice that $\mathfrak p(\mathfrak s(y)) \preccurlyeq \iota(\mathfrak s(y)) \preccurlyeq \mathfrak s(y) \Delta^{-p} = \mathfrak p(y)^{-1} y \mathfrak p(y) \Delta^{-p}$. Hence $\mathfrak p(y) \: \mathfrak p(\mathfrak s(y)) \preccurlyeq y \mathfrak p(y) \Delta^{-p} \preccurlyeq  y^2 \Delta^{-2p}$. Since~$y$ satisfies the required rigidity conditions, Lemma~\ref{L:rigidity_conditions} tells us that $\inf(y^2)=2p$, hence the initial factor of $y^2 \Delta^{-2p}$ is precisely $\iota(y^2)$, which is equal to~$\iota(y)$, again by Lemma~\ref{L:rigidity_conditions}. Since we are assuming that $\mathfrak p(y)\alpha$ is a simple prefix of $\mathfrak p(y) \: \mathfrak p(\mathfrak s(y))$, it follows that $\mathfrak p(y)\alpha \preccurlyeq \iota(y^2) =\iota(y)$.   In the same way, as $\mathfrak p(y^{-1})= \iota(y^{-1})\wedge \iota(y) =\mathfrak p(y)$ one has $\mathfrak p(\mathfrak s(y^{-1}))= \mathfrak p (\mathfrak s(y)^{-1}) =\mathfrak p(\mathfrak s(y))$, we can apply the above argument to~$y^{-1}$ and it follows that $\mathfrak p(y)\alpha \preccurlyeq  \iota(y^{-1})$. Therefore $\mathfrak p(y)\alpha \preccurlyeq \iota(y)\wedge \iota(y^{-1})=\mathfrak p(y)$, so $\alpha=1$, and the first half of the proposition is proven.

Now, if the hypotheses of Proposition~\ref{P:preferred prefix is rigid} are satisfied by~$y$, then by Lemma~\ref{L:rigidity_and_sliding} they are also satisfied by $\mathfrak s^k(y)$ for every~$k>0$. So not only the product $\mathfrak p(y) \: \mathfrak p(\mathfrak s(y))$ is left-weighted as written, but also $\mathfrak p(\mathfrak s^i(y))\: \mathfrak p(\mathfrak s^{i+1}(y))$ is left-weighted for every~$i>0$. Thus the left normal form of $P(y)$ is precisely $\mathfrak p(y) \mathfrak p(y^{(1)}) \cdots \mathfrak p(y^{(N-1)})$, where $N$ is the length of the sliding circuit of $y$. Moreover, as $y^{(N)}=y$, the product $\mathfrak p(y^{(N-1)}) \mathfrak p(y)$ is also left-weighted, hence the left normal form of $P(y)^2$ is $\mathfrak p(y) \mathfrak p(y^{(1)}) \cdots \mathfrak p(y^{(N-1)})\mathfrak p(y) \mathfrak p(y^{(1)}) \cdots \mathfrak p(y^{(N-1)})$, which means that $P(y)$ is rigid.
\end{proof}

Once we have seen that if $\mathcal R(y)>0$ and $\mathcal R(y^{-1})>0$ then~$P(y)$ is rigid, we are interested in finding elements which satisfy these rigidity conditions, so we can gain sufficient control over their preferred conjugator. In the next result,we will see that if $N=||\Delta||^3-||\Delta||^2$, every element in $SC^{[N]}(x)$ has a power which satisfies the required rigidity conditions.

\begin{proposition}\label{P:rigidity of some power}
Let $x\in B_n$, and let $N=||\Delta||^3-||\Delta||^2$. Given $y\in SC^{[N]}(x)$, there is an integer~$m$ with $0<m<N$ such that $\mathcal R(y^m)>0$ and
$\mathcal R(y^{-m})>0$.
\end{proposition}

\begin{proof}
In~\cite{LL_geodesic} it is shown that for every $x\in B_n$ there exists some $k\leqslant ||\Delta||^2$ such that every element in $SSS(x^k)$ is periodically geodesic. That is, for every $z\in SSS(x^k)$ one has $\inf(z^t)=t\cdot\inf(z)$ and $\sup(z^t)=t\cdot\sup(z)$ for all~$t>0$. In particular, since $y\in SC^{[N]}(x)$ and~$k<N$, one has $y^k\in SC(x^k)\subset SSS(x^k)$, so~$y^k$ is periodically geodesic. This means that $\inf(y^{kt})=t\cdot\inf(y^k)$ and $\sup(y^{kt})=t\cdot\sup(y^k)$ for all~$t>0$.

Once $y^k$ is known to be periodically geodesic, one has a chain $\iota(y^k)\preccurlyeq \iota(y^{2k}) \preccurlyeq \iota(y^{3k}) \preccurlyeq \cdots$ (the initial factor of~$y^{ik}$ is a prefix of the initial factor of $y^{(i+1)k}$).  Notice that this chain stabilizes at the first repetition, hence it must stabilize in less than~$||\Delta||$ steps. In the same way, since~$y^k$ is periodically geodesic one has a chain $\cdots \succcurlyeq \varphi(y^{3k}) \succcurlyeq \varphi(y^{2k}) \succcurlyeq \varphi(y^k)$ (the final factor of~$y^{ik}$ is a suffix of the final factor of $y^{(i+1)k}$), which must also stabilize in less than~$||\Delta||$ steps. Therefore, for some $t\leqslant ||\Delta||-1$ one has $\iota(y^{tk})=\iota(y^{2tk})$ and $\varphi(y^{tk})=\varphi(y^{2tk})$. We can take $m=kt\leqslant||\Delta||^3-||\Delta||^2$ and we will have, on the one hand, $\inf(y^{2m})=2\inf(y^m)$ and $\iota(y^{2m})=\iota(y^m)$ (thus $\mathcal R(y^m)>0$), and on the other hand $\sup(y^{2m})=2\sup(y^m)$ and $\varphi(y^{2m})=\varphi(y^m)$ (thus $\mathcal R(y^{-m}))>0$), so the result follows.
\end{proof}

Now we will place ourselves in the case in which a braid $y\in SC(x)$ satisfies the above rigidity conditions, that is, $\mathcal R(y)>0$ and $\mathcal R(y^{-1})>0$ (by Proposition~\ref{P:rigidity of some power} we know how to find a braid which fulfill these requirements). We saw in Proposition~\ref{P:preferred prefix is rigid} that in this case $P(y)$ is rigid. We will now see that, if for some reason we need to consider some power of $y$, this makes no harm, as every power of $y$ satisfies the same properties  (even the property of belonging to a sliding circuit).

\begin{proposition}\label{P:rigidity and preferred conjugators}
Let $x\in B_n$ and $y\in SC(x)$ with $\ell(y)>0$. If $\mathcal R(y)>0$ and $\mathcal R(y^{-1})>0$, then for every $m\geq 1$ one has  $y\in SC^{[m]}(x)$ (that is, $y^m\in SC(x^m)$), $\mathcal R(y^m)>0$, $\mathcal R(y^{-m})>0$, and $P(y)$ is a positive power of $P(y^m)$.
\end{proposition}

\begin{proof}
We recall from~\cite[Proposition 3.9]{BGG1} that if $y\in USS(x)$ and $\ell(y)>0$, then $\mathcal R(y)\leqslant \mathcal R(y^m)$ for all~$m\geqslant 1$. Hence, if $y\in SC(x)$ is such that $\mathcal R(y)>0$ and $\mathcal R(y^{-1})>0$, the same happens for every power of~$y$.

By Lemma~\ref{L:rigidity_conditions}, $\iota(y^m)=\iota(y)$ and $\varphi(y^m)=\varphi(y)$. Hence $\mathfrak p(y^m)=\iota(y^m)\wedge \partial(\varphi(y^m))=\iota(y)\wedge \partial(\varphi(y))=\mathfrak p(y)$.  Therefore $\mathfrak s(y^m)= (\mathfrak s(y))^m$.  By Lemma~\ref{L:rigidity_and_sliding}, $\mathfrak s(y)$ also satisfies the required rigidity conditions, that is, $\mathcal R(\mathfrak s(y))>0$ and $\mathcal R(\mathfrak s(y)^{-1})>0$. Hence $\mathfrak p(\mathfrak s(y)^m) = \mathfrak p(\mathfrak s(y))$ and then $\mathfrak s^2(y^m)=\mathfrak s(\mathfrak s(y^m))= \mathfrak s(\mathfrak s(y)^m)= (\mathfrak s^2(y))^m$ for every~$m>0$.  Iterating this argument, one obtains $\mathfrak p((\mathfrak s^t(y))^m) = \mathfrak p(\mathfrak s^t(y))$ and $\mathfrak s^t(y^m)=(\mathfrak s^t(y))^m$ for every $t,m>0$.  In other words, applying iterated cyclic sliding to~$y^m$ is the same thing as applying iterated cyclic sliding to~$y$ and then taking the~$m$th power, since the conjugating elements coincide. As~$y$ is in a sliding circuit, applying iterated cyclic sliding leads back to~$y$, 
and the same happens to~$y^m$. That is, $y^m$ is also in a sliding circuit, as we wanted to show. Moreover, some positive power of $P(y^m)$ equals $P(y)$ as the preferred prefixes along the circuits of~$y$ and~$y^m$ coincide. Actually, we will have $P(y^m)=P(y)$, unless there is some $z$ in the sliding circuit of $y$ such that $z^m=y^m$, in which case $P(y^m)$ will be shorter than $P(y)$, but continuing along the sliding circuit of $y^m$ one will obtain several repetitions of $P(y^m)$ being equal to $P(y)$.
\end{proof}

We end this section with a result about preferred conjugators which we shall need soon. It says that the preferred conjugators of any two elements in the same set of sliding circuits are conjugate, up to raising those preferred conjugators to some suitable powers. This will allow us to obtain information concerning $P(y)$, for some $y\in SC(x)$, just by comparing $P(y)$ with $P(z)$, for some other $z\in SC(x)$. This time we do not require any rigidity condition.

\begin{lemma}\label{L:preferred conjugators are conjugate}
Let $x\in B_n$ and $y,z\in SC(x)$. Then $P(y)^s$ is conjugate to $P(z)^t$ for some $s,t>0$, and one can take as conjugating element any braid $\alpha$ conjugating $y$ to $z$.
\end{lemma}

\begin{proof}
Let~$N$ and~$M$ be the lengths of the sliding circuits of~$y$ and~$z$, respectively. That is, $\mathfrak s^N(y)=y$ and $\mathfrak s^M(z)=z$.  Let~$\alpha$ be such that $\alpha^{-1}y\alpha=z$. We can apply to~$\alpha$ the transport map defined in~\cite{GG1}. If one applies this transport map~$k$ times to~$\alpha$, we obtain an element denoted~$\alpha^{(k)}$, which is a conjugating element from $\mathfrak s^k(y)$ to $\mathfrak s^k(z)$. Namely,
\begin{equation}\label{eqn_transport}
    \alpha^{(k)}=\left(\mathfrak p(y) \mathfrak p(\mathfrak s(y))\cdots \mathfrak p(\mathfrak s^{k-1}(y)) \right)^{-1} \alpha \; \left(\mathfrak p(z) \mathfrak p(\mathfrak s(z))\cdots \mathfrak p(\mathfrak s^{k-1}(z)) \right).
\end{equation}

   In~\cite[Lemma 8]{GG1} it is shown that, in this situation,  $z\in SC(x)$ if and only if $\alpha^{(sN)}=\alpha$ for some~$s>0$. This means that~$\alpha$ conjugates $\mathfrak s^{sN}(y)=y$ to $\mathfrak s^{sN}(z)$, but since the conjugate of~$y$ by~$\alpha$ is precisely~$z$, it follows that $\mathfrak s^{sN}(z)=z$ hence $sN=tM$ for some~$t>0$. But then Equality~(\ref{eqn_transport}), replacing~$k$ by~$sN$, reads $\alpha = \left(P(y)^s\right)^{-1} \alpha \: P(z)^t$ or, in other words, $\alpha^{-1} P(y)^s \alpha =  P(z)^t$.
\end{proof}


\subsection{Sliding circuits and canonical reduction systems}

\begin{proposition}\label{P:CRS_preferred conjugator}
Let~$x\in B_n$ and $y\in SC(x)$. If $\mathcal R(y)>0$ and $\mathcal R(y^{-1})>0$, then $CRS(P(y)) \subset CRS(y)$.
\end{proposition}

\begin{proof}
Notice that the result holds if~$y$ is periodic, since the only periodic elements satisfying the rigidity hypothesis are powers of~$\Delta$, and then $P(y)=1$, so both canonical reduction systems are empty. If~$y$ is pseudo-Anosov the result also holds, since~$P(y)$ is in the centralizer of~$y$ so it must be either pseudo-Anosov or periodic~\cite{GW}, and in either case $CRS(P(y))=\emptyset=CRS(y)$. We can then assume that~$y$ is non-periodic and reducible, that is, $CRS(y)\neq \emptyset$. And of course we can assume that $CRS(P(y))\neq \emptyset$, otherwise the result is trivially true.

By Proposition~\ref{P:rigidity and preferred conjugators}, we can make the further assumption that~$y$ is pure, since~$y^m$ will satisfy the same hypothesis as~$y$, and the canonical reduction systems of $y$ and of its preferred conjugator are preserved by taking powers of $y$. Replacing~$P(y)$ by a power if necessary in the following discussion, we will also assume that~$P(y)$ is pure.

Let $\mathcal F=CRS(y)\cup \{\partial(D^2)\}$, and let us assume for a moment that all curves in~$\mathcal F$ are round. Since~$P(y)$ is pure and commutes with~$y$, $P(y)$~sends $\mathcal F$ to itself, curve-wise. This implies that an essential reduction curve of $P(y)$ either belongs to $\mathcal F$ (as we want to show) or can be isotoped to be disjoint from $\mathcal F$. In the latter case, it would correspond to an essential reduction curve of $P(y)_{[\mathcal C\in \mathcal F]}$ for some $[\mathcal C]\in \mathcal F$.  Thus we must show that $P(y)_{[\mathcal C\in \mathcal F]}$ does not admit an essential reduction curve, for every $[\mathcal C]\in \mathcal F$.

Let then $[\mathcal C]\in \mathcal F$. We know that $y_{[\mathcal C\in \mathcal F]}$ is either periodic or pseudo-Anosov, and that the braid $P(y)_{[\mathcal C\in \mathcal F]}$ commutes with $y_{[\mathcal C\in \mathcal F]}$. If $y_{[\mathcal C\in \mathcal F]}$ is pseudo-Anosov, then $P(y)_{[\mathcal C\in \mathcal F]}$ must be either pseudo-Anosov or periodic, hence it admits no essential curves. If $y_{[\mathcal C\in \mathcal F]}$ is periodic, it has to be a power of the full twist, since~$y$ is pure. But in this case Proposition~\ref{P:preferred_prefix_round_curves} tells us that $\mathfrak p(y)_{[\mathcal C\in \mathcal F]}$ is either trivial or a half twist (here we use that $\mathcal F$ consists of round curves). Hence, applying cyclic sliding to~$y$, we obtain a braid whose component associated to~${\mathcal C}$ is also a power of the half twist, and we can repeat the argument until one gets back to~$y$, to conclude that $P(y)_{[\mathcal C\in \mathcal F]}$ is a (possibly trivial) power of~$\Delta$. Hence $P(y)_{[\mathcal C\in \mathcal F]}$ does not admit an essential curve, also in this case.
Therefore, all essential reduction curves of $P(y)$ are essential curves of $y$, that is, $CRS(P(y))\subset CRS(y)$ if  $CRS(y)$ is a family of round curves.

Now we show the general case, in which the curves in $CRS(y)$ are not necessarily round. We cannot apply the above argument as we do not know, a priori, that the components of $P(y)$ corresponding to the periodic components of $y$ are powers of $\Delta$. Nevertheless, we will be able to show this by comparing preferred prefixes with the aid of Lemma~\ref{L:preferred conjugators are conjugate}. We just need to find a suitable braid whose reduction curves are round and which satisfies the hypothesis of Proposition~\ref{P:CRS_preferred conjugator}, that is, it belongs to a sliding circuit, and both the braid and its inverse have nonzero rigidity.

By Corollary~\ref{C:round curves in SC^m}, for every~$N>0$ there is some element $z\in SC^{[N]}(y)$ whose essential curves are all round. We can then take $N=||\Delta||^3-||\Delta||^2$ and use Proposition~\ref{P:rigidity of some power} to conclude that for some~$m$ with $0<m\leqslant N$ we have $\mathcal R(z^m)>0$ and $\mathcal R(z^{-m})>0$. As $m\leq N$, we also have $z^m\in SC(y^m)$. Notice that the canonical reduction systems of~$z$ and~$z^m$ coincide, so~$z^m$ is a braid whose canonical reduction system is made of round curves, which belongs to a sliding circuit, and such that $\mathcal R(z^m)>0$ and $\mathcal R(z^{-m})>0$, so $z^m$ is the braid we were looking for. To simplify notation, we recall from Proposition~\ref{P:rigidity and preferred conjugators} that the result will be shown for $y$ if it is shown for $y^m$, so we can replace $y$ by $y^m$, and this will replace $z$ by $z^m$. We can then assume that $z$ is a braid whose canonical reduction system is made of round curves, which belongs to a sliding circuit, and such that $\mathcal R(z)>0$ and $\mathcal R(z^{-1})>0$.


As the result is shown for elements whose canonical reduction system is made of round curves, $CRS(P(z))\subset CRS(z)$. But recall from Lemma~\ref{L:preferred conjugators are conjugate} that $P(z)^s$ is conjugate to $P(y)^t$ for some~$s,t>0$, and that a conjugating element~$\alpha$ is precisely a conjugating element from~$z$ to~$y$. Since the essential curves of $P(z)$ and $P(z)^s$ coincide, we have $CRS(P(z)^s)=CRS(P(z))\subset CRS(z)$. Conjugating both $P(z)^s$ and~$z$ by~$\alpha$, corresponds to applying~$\alpha$ to their essential curves, hence it follows that $CRS(P(y)^t)\subset CRS(y)$. As the essential curves of $P(y)^t$ and $P(y)$ coincide, this means $CRS(P(y))\subset CRS(y)$, as we wanted to show.
\end{proof}

We have now assembled most of the ingredients for showing that our main
result, Theorem~\ref{T:main}, follows from the rigid case. The key lemma
for this reduction to the rigid case is as follows.

\begin{lemma}\label{L:main_detailed}
Let $x\in B_n$ be a non-periodic, reducible braid. Let $N=||\Delta||^3-||\Delta||^2$. For every element $y\in SC^{[N]}(x)$ there is some $m\leqslant N$ such that either $y^m$ is rigid, or $P(y^m)$ is rigid, admits essential reduction curves, and all its essential reduction curves are essential reduction curves of~$y$.
\end{lemma}

\begin{proof}
Let $x\in B_n$ be a non-periodic, reducible braid, $N=||\Delta||^3-||\Delta||^2$ and $y\in SC^{[N]}(x)$. By Proposition~\ref{P:rigidity of some power} there is some power~$y^m$ with~$m\leqslant N$ such that $\mathcal R(y^m)>0$ and $\mathcal R(y^{-m})>0$. Notice also that $y^m\in SC(x^m)$. Hence~$y^m$ satisfies the hypothesis of Propositions~\ref{P:preferred prefix is rigid} and~\ref{P:CRS_preferred conjugator}, so $P(y^m)$ is rigid and $CRS(P(y^m))\subset CRS(y^m)$. If $CRS(P(y^m))\neq\emptyset$, the result follows.

Suppose on the contrary that $CRS(P(y^m))=\emptyset$. This means that $P(y^m)$ must be either periodic or pseudo-Anosov. It cannot be pseudo-Anosov, as it commutes with the non-periodic, reducible braid~$y^m$, while pseudo-Anosov elements can only commute with pseudo-Anosov or periodic ones. Hence $P(y^m)$ is periodic. Notice that $P(y^m)$ cannot be a nontrivial power of~$\Delta$, since by Proposition~\ref{P:preferred prefix is rigid} the left normal form of $P(y^m)$ is a product of preferred prefixes, each of them not equal to $\Delta$ by definition. As the only rigid, periodic braids are the powers of $\Delta$, it follows that $P(y^m)$ must be trivial. This is equivalent to saying that~$y^m$ is rigid.
\end{proof}

The following result tells us how to deal with the rigid case.
We will assume for the moment; it will be shown in the next section:

\begin{theorem}\label{T:rigid_case}
Let $\beta\in B_n$ be a non-periodic, reducible braid which is rigid. Then there is some positive integer $k\leqslant n$ such that one of the following conditions holds:
\begin{enumerate}
\item $\beta^k$ preserves a round essential curve, or

\item $\inf(\beta^k)$ and $\sup(\beta^k)$ are even, and either $\Delta^{-\inf(\beta^k)}\beta^k$ or $\beta^{-k}\Delta^{\sup(\beta^k)}$ is a positive braid which preserves an almost round essential reduction curve whose corresponding interior strands do not cross.
\end{enumerate}
In particular, some essential reduction curve for $\beta$ is either round or almost round.
\end{theorem}


We can finally show our main result, assuming that Theorem~\ref{T:rigid_case} holds.

\begin{proof}[Proof of Theorem~\ref{T:main}]
Let $x\in B_n$ be a non-periodic, reducible braid, $N=||\Delta||^3-||\Delta||^2$ and $y\in SC^{[N]}(x)$. Let $m\leqslant N$ be the integer given by Lemma~\ref{L:main_detailed}. If $y^m$ is rigid, then by Theorem~\ref{T:rigid_case} $CRS(y^m)$ contains a curve which is either round or almost round. As $CRS(y^m)=CRS(y)$, the result follows in this case.

If $y^m$ is not rigid, then by Lemma~\ref{L:main_detailed}, $P(y^m)$ is rigid and $\emptyset \neq CRS(P(y^m))\subset CRS(y^m)=CRS(y)$. By Theorem~\ref{T:rigid_case} again, some curve in $CRS(P(y^m))$, and thus in $CRS(y)$, is either round or almost round.

This shows that every element in $SC^{[N]}(x)$ admits an essential reduction curve which is either round or almost round. This implies the result.
\end{proof}


\section{Reducible rigid braids}\label{S:reducible rigid braids}

This section is devoted to the proof of Theorem~\ref{T:rigid_case}.

Let $\beta\in B_n$ be a non-periodic, reducible braid which is rigid. Then~$\beta$ belongs to a sliding circuit (as $\mathfrak s(\beta)=\beta$), also $\ell(\beta)>0$ and $CRS(\beta)\neq \emptyset$. Also, any power~$\beta^k$ of~$\beta$ is also non-periodic, reducible and rigid, and has the same canonical reduction system as~$\beta$. Notice that for every curve $\mathcal C\in CRS(\beta)$, there is some $t\leqslant n/2$ such that $[\mathcal C]^{\beta^t}=[\mathcal C]$. Replacing $\beta^t$ by its square if necessary, it follows that for every $\mathcal C\in CRS(x)$ there is some even $k\leqslant n$ such that~$\beta^k$ preserves~$[\mathcal C]$, and both $\inf(\beta^k)$ and $\sup(\beta^k)$ are even.

Fix an innermost curve $\mathcal C\in CRS(\beta)$ and consider $\beta^k$ for some even $k\leqslant n$ such that $[\mathcal C]^{\beta^k}=[\mathcal C]$. Let $\Delta^{2p} x_1\cdots x_r$ be the left normal form of~$\beta^k$, and denote $x=x_1\cdots x_r=\Delta^{-\inf(\beta^k)}\beta^k$. Notice that $CRS(\beta)=CRS(\beta^k)=CRS(\Delta^{2p}x)=CRS(x)$, as~$\Delta^2$ preserves every simple closed curve of the punctured disc. Moreover $x=x_1\cdots x_r$ is non-periodic, reducible and rigid.

Denote $\mathcal F=CRS(x)=CRS(\beta)\neq \emptyset$. As $\mathcal C$ is an innermost curve of $\mathcal F$, the component $x_{[\mathcal C\in \mathcal F]}$ must be either periodic or pseudo-Anosov. Recall that in order to define $x_{[\mathcal C\in \mathcal F]}$ one conjugates $x$ by the minimal standardizer of $\mathcal F$ to obtain $y=\widehat x$, and the curve corresponding to $\mathcal C$, namely $\widehat{\mathcal C}$, is an innermost essential curve of $y$ which is round. By~\cite[Theorem 4.9]{LL} $y$ belongs to its Ultra Summit Set provided $x$ does. It is not difficult to modify the proof in~\cite{LL} to show that $y$ belongs to $SC(x)$ provided $x$ does. This is the case, as $x$ is rigid. But it is shown in~\cite{GG1} that, if $x$ is rigid,  $SC(x)$ consists precisely of the rigid conjugates of $x$. Hence $y$ is rigid. Moreover, as $x$ preserves $[\mathcal C]$, $y_{[\widehat{\mathcal C}\in \widehat{\mathcal F}]}$ is a conjugate of $x_{[\mathcal C\in \mathcal F]}$, which is either periodic or pseudo-Anosov.

Suppose $y_{[\widehat{\mathcal C}\in \widehat{\mathcal F}]}$ is periodic. As~$y$ is a rigid, positive braid, whose left normal form has the form $y_1\cdots y_r$, one has that $y_ry_1$ is left weighted as written. But the left normal form of $y_{[\widehat{\mathcal C}\in \widehat{\mathcal F}]}$ is determined by the left normal form of~$y$, in the sense explained in Lemma~\ref{L:lnf_round_curves}. Hence  $y_{[\widehat{\mathcal C}\in \widehat{\mathcal F}]}$ must be a rigid, positive braid whose left normal form is the product of~$r$ (possibly trivial) simple elements.  Since the only periodic rigid elements are powers of~$\Delta$, it follows that either $y_{[\widehat{\mathcal C}\in \widehat{\mathcal F}]}$ is trivial, or $y_{[\widehat{\mathcal C}\in \widehat{\mathcal F}]}=\Delta_k^r$ (where~$k$ is the number of strands inside~$\widehat{\mathcal C}$).

If $y_{[\widehat{\mathcal C}\in \widehat{\mathcal F}]}$ is trivial, the interior braid of~$x=\Delta^{-\inf(\beta^m)}\beta^m$ associated to~$\mathcal C$ must also be trivial, as it is a conjugate of $y_{[\widehat{\mathcal C}\in \widehat{\mathcal F}]}$. By Proposition~\ref{P:trivial interior -> almost round}, $\mathcal C$ is either round or almost round, so Theorem~\ref{T:rigid_case} holds in this case.

Suppose that $y_{[\widehat{\mathcal C}\in \widehat{\mathcal F}]}=\Delta_k^r$, and notice that $r=\sup(\beta^m)-\inf(\beta^m)$ is even. Let us consider the $n$-strand braids~$x'$ and~$y'$ such that $xx'=\Delta^r$ and $yy'=\Delta^r$. We remark that~$x'$ and~$y'$ are basically the inverses of $x$ and $y$, multiplied by some even power of $\Delta$ so that their infimum becomes 0. Hence $x'$ and $y'$ are positive, rigid braids of infimum 0 and canonical length~$r$, whose canonical reduction systems coincide with those of~$x$ and~$y$, respectively. Let $\alpha$ be such that $\alpha^{-1}x\alpha=y$. Since $\alpha^{-1}\Delta^r\alpha =\Delta^r$ as~$r$ is even, we obtain that $\alpha^{-1} x' \alpha = y'$. Moreover, $y'_{[\widehat{\mathcal C}\in \widehat{\mathcal F}]}$ is trivial. Hence, the strands of~$x'$ interior to~$\mathcal C$ do not cross. By Proposition~\ref{P:trivial interior -> almost round}, $\mathcal C$ is either round or almost round. Now notice that $x'=x^{-1}\Delta^r= \beta^{-m} \Delta^{\inf(\beta^m)+r} = \beta^{-m}\Delta^{\sup(\beta^m)}$. Hence Theorem~\ref{T:rigid_case} also holds in this case.

It only remains to prove Theorem~\ref{T:rigid_case} in the case in which $x_{[\mathcal C\in \mathcal F]}$ is pseudo-Anosov.

\begin{lemma}\label{L:convexity}
Let $x\in B_n$. Given two elements $y,z\in SC(x)$, there is a sequence of conjugations
$$
   y = \alpha_1 \stackrel{s_1}{\longrightarrow} \alpha_2 \stackrel{s_2}{\longrightarrow} \alpha_3 \cdots
   \stackrel{s_r}{\longrightarrow} \alpha_{r+1} =z
$$
such that for every $i=1,\ldots,r$  one has $\alpha_{i+1}=\alpha_{i}^{s_i}\in SC(x)$, and  either $s_i
\preccurlyeq \iota(\alpha_i)$ or $s_i \preccurlyeq \iota(\alpha_i^{-1})$.
\end{lemma}

\begin{proof}
This proof follows the ideas in~\cite{EM,FG,BGG2}. First, we can assume that $\ell(y)>0$, otherwise $SC(x)=\{\Delta^p\}$ for some~$p$, and the result becomes trivial as~$y=z$. Now~$y$ and~$z$ are conjugate since they belong to $SC(x)$. Multiplying any conjugating element by a sufficiently large power
of~$\Delta$, it follows that $z=y^\alpha$ for some positive element~$\alpha$. This conjugating element~$\alpha$ can obviously be decomposed into a product of {\it indecomposable} conjugating elements, that is, $\alpha=s_1\cdots s_r$, where $\alpha_{i+1}=y^{s_1\cdots s_i}\in SC(x)$ for $i=1,\ldots,r$, and~$s_i$ is positive and cannot be decomposed as a product of two nontrivial positive elements $s_i=ab$ such that $\alpha_i^{a}\in SC(x)$. Notice that $s_i$ must be simple, otherwise we could take $a=s_i\wedge \Delta$ (which by Theorem~\ref{T:convexity} satisfies $\alpha_i^a\in SC(x)$) to decompose $s_i$. We must show that such an indecomposable element~$s_i$ must be a prefix of either $\iota(\alpha_i)$ or $\iota(\alpha_i^{-1})$.

Denote $t=s_i\wedge \iota(\alpha_i^{-1})$. We claim that $(\alpha_i)^{t}\in SC(x)$. Indeed, by definition, one has
$\iota(\alpha_i^{-1})=  \Delta \wedge (\alpha_i^{-1} \Delta^{-\inf(\alpha_i^{-1})} )$. Since $\alpha_i\in SC(x)$, it is clear that $\alpha_i^{\Delta}\in SC(x)$ and that
$\alpha_i^{(\alpha_i^{-1} \Delta^{-\inf(\alpha_i^{-1})})}\in SC(x)$. By Theorem~\ref{T:convexity}, $\alpha_i^{\iota(\alpha_i^{-1})}= \alpha_i^{\Delta \wedge (\alpha_i^{-1} \Delta^{-\inf(\alpha_i^{-1})} )} \in SC(x)$. But $\alpha_i^{s_i}=\alpha_{i+1}\in SC(x)$, so applying Theorem~\ref{T:convexity} again one has $\alpha_i^{\iota(\alpha_i^{-1})\wedge s_i}=(\alpha_i)^t\in SC(x)$, as we wanted to show.

We then have a positive prefix $t\preccurlyeq s_i$ such that $(\alpha_i)^t\in SC(x)$. Since~$s_i$ is an indecomposable conjugator, it follows that either $t=s_i$ or~$t=1$. In the former case $s_i=t=s_i\wedge \iota(\alpha_i^{-1})$, which implies $s_i\preccurlyeq \iota(\alpha_i^{-1})$, hence the result holds in this case.

Suppose then that~$t=1$. This means $\iota(\alpha_i^{-1})\wedge s_i = \partial(\varphi(\alpha_i)) \wedge s_i=1$, which is equivalent to say that $\varphi(\alpha_i) s_i$ is left weighted as written. Let $\Delta^p a_1\cdots a_r$ be the left normal form of~$\alpha_i$. We have then shown that $a_rs_i$ is left weighted as written, so $\Delta^p a_1\cdots a_r s_i$ is the left normal form of $\alpha_i s_i$. But we know that $\alpha_{i+1}=s_i^{-1}\alpha_is_i\in SC(x)$. In particular $\ell(\alpha_{i+1})=r$, where $\alpha_{i+1}=s_i^{-1}\Delta^p a_1\cdots a_r s_i$. This implies that $\tau^p(s_i)\preccurlyeq a_1\cdots a_r s_i$, where the right hand side is in left normal form and the left hand side is a simple element, hence $\tau^p(s_i)\preccurlyeq a_1$, that is, $s_i\preccurlyeq \tau^{-p}(a_1)=\iota(\alpha_i)$, so the result also holds in this case.
\end{proof}

Finally, here is the result that completes the proof of Theorem~\ref{T:rigid_case}:

\begin{proposition}\label{P:rigid_and_PA_interior}
Let~$x$ be a reducible rigid braid, and let~$\mathcal C$ be an invariant curve of~$x$
whose corresponding interior braid is pseudo-Anosov. Then~$\mathcal C$ is round.
\end{proposition}

\begin{proof}
We know from~\cite{GG1} that $SC(x)$ is the set of rigid conjugates of~$x$, hence $x\in SC(x)$, and we know
from Corollary~\ref{C:round curves in SC^m} that there is an element $\tilde x\in SC(x)$ whose reduction curves are all round.
By Lemma~\ref{L:convexity} there is a chain of conjugations
$$
   \tilde x = \alpha_1 \stackrel{t_1}{\longrightarrow} \alpha_2 \stackrel{t_2}{\longrightarrow} \alpha_3 \cdots
   \stackrel{t_r}{\longrightarrow} \alpha_{r+1} =x
$$
such that for every $i=1,\ldots,r$  one has $\alpha_{i+1}=\alpha_{i}^{t_i}\in SC(x)$, and  either $t_i
\preccurlyeq \iota(\alpha_i)$ or $t_i \preccurlyeq \iota(\alpha_i^{-1})$.

Suppose that~$\mathcal C$ is not round. This means that the curve $\mathcal C_{\tilde x}$ of~$\tilde x$ corresponding to~$\mathcal C$ is a round curve which loses its roundness after the application of $t_1\cdots t_r$. This implies that there must be two rigid braids $y,z\in SC(x)$ (precisely~$\alpha_i$ and~$\alpha_{i+1}$ for some~$i$), conjugate by a simple element~$s$ (precisely $t_i$), a round invariant curve~$\mathcal C_y$ of~$y$ whose corresponding interior braid is pseudo-Anosov, and the corresponding invariant curve of~$z$, $[\mathcal C_z]=[\mathcal C_y]^s$, which is not round. Moreover~$s$ is either a prefix of $\iota(y)$ or a prefix of $\iota(y^{-1})$ (as $s=t_i$). Since the inverse of a pseudo-Anosov braid is also pseudo-Anosov, and the rigidity and reduction
curves of a braid are preserved by taking inverses, we can replace~$y$ and~$z$ by~$y^{-1}$ and~$z^{-1}$ if necessary, so we can assume that~$s$ is a prefix of
$\iota(y^{-1})$.

Since taking powers and multiplying rigid braids by~$\Delta^{2k}$ are operations which do not affect their rigidity, their initial factors, their invariant curves or the geometric type of their corresponding interior braids, we can further assume that~$y$ and~$z$ are pure braids, and that $\inf(y)=\inf(z)=0$.

Suppose that some nontrivial positive prefix $s'\preccurlyeq s$ is such that $[\mathcal C_y]^{s'}$ is round, and
denote by~$\rho$ the minimal positive element such that $s'\preccurlyeq \rho$ and
$y^\rho$ is rigid (equivalently, $y^\rho\in SC(x)$). Since~$s$ is an indecomposable conjugator, we must have $\rho=s$. But we will now see that $\rho$ sends $[\mathcal C_y]$ to a round curve, while $[\mathcal C_y]^\rho =[\mathcal C_y]^s=[\mathcal C_z]$ is not round. A contradiction that will imply that~$s'=1$.
Indeed, by~\cite[Algorithm 2, step 3(b)]{GG2}, $\rho$ can be computed in the following way: first, while $y^{s'}\notin SSS(x)$, replace~$s'$  by
$$
s'\cdot \left(1\ \vee \ (y^{s'})^{-1}\Delta^{\inf y}\ \vee \ y^{s'} \Delta^{-\sup y}\right).
$$
Notice that the three elements $1$, $(y^{s'})^{-1}\Delta^{\inf y}$ and~$y^{s'}
\Delta^{-\sup y}$ send $[\mathcal C_y]^{s'}$ to a round curve. In the terminology
of~\cite{LL}, the three elements belong to the {\it standardizer} of $[\mathcal C_y]^{s'}$.
Since it is shown in~\cite{LL} that the standardizer of a curve is closed under~$\vee$,
it follows that each step of this procedure replaces~$s'$ by a bigger element, which
belongs to the standardizer of $[\mathcal C_y]$. Hence we can assume that $y^{s'}\in
SSS(x)$. The second step to compute~$\rho$, explained in~\cite[Theorem 2]{GG1}, consists of applying
iterated sliding to~$y^{s'}$ until one reaches a rigid element. Multiplying~$s'$ on the
right by all conjugating elements, one obtains~$\rho$. But each conjugating element for
sliding maintains the roundness of our distinguished curve, from Proposition~\ref{P:sliding_preserves_roundness}. Therefore, $\rho$ sends~$\mathcal C_y$ to a round curve, but $[\mathcal C_y]^\rho = [\mathcal C_z]$ is not round. A contradiction. It follows that~$s'=1$, or in other words, there is no nontrivial prefix $s'\preccurlyeq s$ is such that $[\mathcal C_y]^{s'}$ is round.

Let $p, p+1,\ldots,q$ be the punctures inside~$\mathcal C_y$. We will collect the strands of~$s$ into three sets, $L=\{1,\ldots,p-1\}$, $I=\{p,p+1,\ldots,q\}$ and $R=\{q+1,q+2,\ldots,n\}$, depending whether they start to the left, inside or to the right of~$\mathcal C_y$. Since every prefix of~$s$ must deform the round curve~$\mathcal C_y$, and the braid~$s$ is simple, it follows that the strands in~$L$ (resp. in~$I$ and in~$R$) do not cross each
other in~$s$, since this would imply that two consecutive strands in~$L$ (resp. in~$I$ and in~$R$) would cross in~$s$, and the corresponding crossing would be a prefix of~$s$ preserving the roundness of~$\mathcal C_y$, a contradiction. Also, no strand of~$s$ in~$L$ can cross {\it all} the strands in~$I$, since this would imply that the strand $p-1$ would cross all the strands in~$I$, and then $\sigma_{p-1}\sigma_p \cdots \sigma_{q-1}$ would be a prefix of~$s$ preserving the roundness of~$\mathcal C_y$, a contradiction. In the same way, no strand of~$s$ in~$R$ can cross {\it all} the strands in~$I$. In summary, $s$ is a simple braid of a very particular form: some strands of~$L$ may cross some (but not all) strands of~$I$,  some strands of~$R$ may cross some (but not all) strands of~$I$, and any two strands belonging to the same group ($L$, $I$, or $R$) never cross.

Recall that~$y$ and~$z$ are rigid, and let $y_1\cdots y_r$ and $z_1\cdots z_r$ be their left normal forms. For $i=0\ldots,r$, we denote $[\mathcal C_{y,i}]=[\mathcal C_y]^{y_1\cdots y_i}$ and $[\mathcal C_{z,i}]=[\mathcal C_z]^{z_1\cdots z_i}$. By Theorem~\ref{T:BGN}, $\mathcal C_{y,i}$ is round for every~$i$, and by the rigidity of~$z$ it follows that $\mathcal C_{z,i}$ is not round for any~$i$.  Now, for $i=0,\ldots,r$, consider the braid $s_i=(y_i^{-1}\cdots y_1^{-1})s(z_1\cdots z_i)$, which is the~$i$th transport of~$s$ under cycling (see~\cite{Gebhardt}). Notice that $s_0=s_r=s$. As transport under cycling preserves prefixes, products, greatest common divisors, and the positivity and simplicity of braids~\cite{Gebhardt}, it follows that~$s_i$ is a simple element for $i=0\ldots, r$. Hence~$s_i$ is a simple element that conjugates the rigid braid $y_{i+1}\cdots y_r y_1\cdots y_i$ to the rigid braid $z_{i+1}\cdots z_rz_1\cdots z_i$, and sends the round curve $\mathcal C_{y,i}$ to the non-round curve $\mathcal C_{z,i}$.

We claim that for $i=0,\ldots,r$, the element~$s_i$ is an indecomposable conjugator. Indeed, if this is not the case, we have a decomposition $s_i=a_ib_i$, where~$a_i$ and~$b_i$ are nontrivial simple braids, such that $a_i^{-1}(y_{i+1}\cdots y_r y_1\cdots y_i)a_i$ is a rigid braid, whose left normal form will have the form $w_{i+1}\cdots w_r w_1\cdots w_i$. We can apply transport under cycling to $a_i$~\cite{Gebhardt} and we obtain $w_{i+1}^{-1} a_i y_{i+1}$. As $a_i$ is conjugating a rigid braid to another rigid braid, the second transport of $a_i$ will be $a_i^{(2)}= w_{i+2}^{-1} w_{i+1}^{-1} a_i y_{i+1} y_{i+2}$. Iterating this process, it follows that the~$r$th transport of $a_i$ under cycling will be $a_i^{(r)}=(w_i^{-1}\cdots w_1^{-1}w_r^{-1}\cdots w_{i+1}^{-1}) a_i (y_{i+1}\cdots y_ry_1\cdots y_i)=a_i$. So~$a_i$ is preserved by~$r$th transport. This implies that no transport of~$a_i$ can be trivial (since the transport of the trivial braid is trivial). In particular the $(r-i)$th transport of~$a_i$, that we will denote~$a_r$, is not trivial. In the same way, the~$r$th transport of~$b_i$, that we will denote~$b_r$, is not trivial. Hence, by the properties of transport, $a_r$ and~$b_r$ are simple, nontrivial braids such that $y^{a_r}$ is rigid, and $a_rb_r=s$ (as $a_ib_i=s_i$ and transport preserves products). This contradicts the indecomposability of~$s$, and shows that~$s_i$ must be an indecomposable conjugator for $i=0,\ldots,r$.

Recall that~$s_i$ sends the round curve~$\mathcal C_{y,i}$ to the non-round curve~$\mathcal C_{z,i}$. As~$s_i$ is an indecomposable conjugator, an argument analogous to the one we used for~$s$, tells us that no prefix of~$s_i$ can send~$\mathcal C_{y,i}$ to a round
curve. Hence, each~$s_i$ is a very special simple braid which has the same form, with respect to~$\mathcal C_{y,i}$, as~$s$ has with respect
to~$\mathcal C_y$.

Notice that we have the equalities:
\begin{equation}\label{E:ys}
  (y_1\cdots y_r)s \; = \; (y_1\cdots y_{r-1})s_{r-1}(z_r)  \; =\cdots = \; (y_1\cdots y_i)s_i (z_{i+1}\cdots z_r)  \; = \cdots =  \; s (z_1\cdots z_r).
\end{equation}
We are now going to deal with the strands in the factors $y_1,\ldots,y_r$ and $s_0,\ldots,s_r$. In order to avoid confusion, we will refer to the strands in each of these factors by the position they have at the beginning of the braid $y_1\cdots y_r s$, or any of the alternative factorizations shown in~(\ref{E:ys}). For instance, if we refer to the strand~$k$ of~$s_i$, we mean the strand of~$s_i$ which starts at position~$k$ at the beginning of $(y_1\cdots y_i)s_i (z_{i+1}\cdots z_r)$. Notice that, as~$y$ is pure, there is no ambiguity with the names of the strands of $s=s_0=s_r$.

Since $[\mathcal C_z]=[\mathcal C_y]^s$ is not round, some strand of~$s$ in either~$L$ or~$R$ must cross some strand in~$I$. Suppose that some strand in~$L$ does (the other case is symmetric). Then the rightmost strand of~$s$ in~$L$, that is, the strand $p-1$ of~$s$, must cross some strands in~$I$. Let us define the set~$I_0$ to be the set of strands in~$I$ which are crossed in~$s$ by the strand~$p-1$. In the same way, define for $i=1,\ldots,r$, the set~$I_i$ to be the set of strands inside~$\mathcal C_{y,i}$ which are crossed, in~$s_i$, by the rightmost strand that starts to the left of~$\mathcal C_{y,i}$.

We will show that $I_i\subset I_{i+1}$ for $i=0,\ldots,r-1$. Indeed, since~$y$ is rigid, one has $\iota(y)\wedge \iota(y^{-1})=y_1\wedge \iota(y^{-1})=1$, and since $s\preccurlyeq \iota(y^{-1})$, one also has $y_1\wedge s=1$ (applying transport under cycling to this equality one has $y_i\wedge s_{i-1}=1$ for all~$i$). Now the strand $p-1$ of~$s$ crosses some strands in~$I$, in particular, the strands~$p-1$ and~$p$ cross in~$s$, hence they do not cross in~$y_1$. As~$y_1$ preserves the roundness of~$\mathcal C_y$, this implies that the strand~$p-1$ of~$y_1$ crosses no strand in~$I$ (since either it crosses all of them or it crosses none).

We claim that $y s\wedge \Delta=y_1s_1$. Indeed, by definition $s_1=y_1^{-1}sz_1$, that is $y_1s_1=sz_1$. Recall that $s\preccurlyeq \iota(y^{-1})=\partial(y_r)$, which means that~$y_rs$ is simple. The transport under cycling (based at $y_ry_1\cdots y_{r-1}$) of this simple braid is precisely $y_1s_1$, so $y_1s_1=sz_1$ is simple. Then $y_1s_1= y_1s_1\wedge \Delta = sz_1\wedge \Delta = s(z\wedge \Delta )\wedge \Delta = sz \wedge s\Delta \wedge \Delta = sz \wedge \Delta = ys \wedge \Delta$, showing the claim.

Now recall that the strand $p-1$ of~$y_1$ crosses no strand in~$I$. As $s^{-1}y s = z$ is a positive braid, $s\preccurlyeq y s$. Since
$s$ is simple, $s\preccurlyeq y s \wedge \Delta = y_1 s_1$ by the above claim. This means that the strands crossed by $p-1$ in~$s$ must also be crossed in $y_1s_1$ but they are not crossed in~$y_1$, hence they are crossed in $s_1$. Notice that the strand $p-1$ of~$s$ does not need to be the rightmost strand to the left of~$\mathcal C_{y,1}$ at the beginning of~$s_1$. Nevertheless, since neither the strands to the left of~$\mathcal C_{y,1}$ nor the strands inside~$\mathcal C_{y,1}$ cross in~$s_1$, if some of the strands to the left of~$\mathcal C_{y,1}$ crosses an interior strand, then the rightmost strand to the left of~$\mathcal C_{y,1}$ also crosses it. Therefore, the strands which are crossed by $p-1$ in~$s$ are crossed by the rightmost strand to the left of~$\mathcal C_{y,1}$ in~$s_1$. In other words,~$I_0\subset I_1$. Applying the same argument to the rigid braids $y_{i}\cdots y_r y_1\cdots y_{i-1}$ for $i=2,\ldots,r$, it follows that $I_0\subset I_1\subset \cdots \subset I_r$. Since by definition $I_r=I_0$, we have the equality $I_i=I_j$ for all $i,j\in \{0,\ldots,r\}$.

Notice that $0< \#(I_0)< q-p-1$, that is,~$I_0$ contains some interior strands but not all of them. Let us now define $J_0=I\backslash I_0$. We will see that no strand of~$I_0$ crosses a strand of~$J_0$ in the whole braid $y$.  Indeed, since the strands in $I_i=I_0$ are the strands in~$s_i$ crossed by the rightmost strand to the left of~$\mathcal C_{y,i}$, it follows that they are the leftmost $\#(I_0)$ strands inside~$\mathcal C_{y,i}$ at the beginning of each $s_i$, that is, at the end of each $y_{i-1}$. Therefore, for $i=1,\ldots,r$, the leftmost $\#(I_0)$ strands inside~$\mathcal C_{y,i}$ at the end of each $y_i$ are always the same, meaning that they never cross in~$y$ with the other interior strands, that is, with the strands in~$J_0$. But this implies that the interior braid of~$y$ corresponding to~$\mathcal C_y$ is {\it split}, that is, the generator $\sigma_{\#(I_0)}$ does not appear in any positive word representing that interior braid. This contradicts the fact that the interior braid is pseudo-Anosov, since a pseudo-Anosov braid can never be split.
\end{proof}


\end{document}